\documentclass[11pt]{amsart}
\usepackage{preamble}

\usepackage[headheight=15pt, headsep=15pt, footskip=27pt, bottom=2.4cm, left=2.4cm, right=2.4cm]{geometry}

\begin{document}

\title[On the relative Morrison-Kawamata cone conjecture (II)]{On the relative Morrison-Kawamata cone conjecture (II)}


\begin{abstract}
Assuming the Morrison-Kawamata cone conjecture for the generic fiber of a Calabi-Yau fibration and the abundance conjecture, we show (1) the finiteness of minimal models, (2) the existence of a weak rational polyhedral fundamental domain under the action of birational automorphism groups, and (3) the finiteness of varieties as targets of contractions. As an application, the finiteness of minimal models and the weak Morrison-Kawamata cone conjecture in relative dimensions $\leq 2$ are established.
\end{abstract}

\author{Zhan Li}
\address[]{Department of Mathematics, Southern University of Science and Technology, 1088 Xueyuan Rd, Shenzhen 518055, China} \email{lizhan.math@gmail.com, lizhan@sustech.edu.cn}

\maketitle

\tableofcontents

\section{Introduction}

Compared with varieties of general type and Fano varieties, the birational geometry of Calabi-Yau varieties (those with trivial canonical divisors) poses the most challenge to study. The standard conjecture of the minimal model program predicts that a variety with intermediate Kodaira dimension birationally admits a Calabi-Yau fibration. The birational geometry of Calabi-Yau fibrations is largely prescribed by the Morrison-Kawamata cone conjecture \cite{Mor93, Mor96, Kaw97, Tot09}.

Let $\Gamma_B$ and $\Gamma_A$ be the images of the pseudo-automorphism group $\PsAut(X/S,\De)$ and the automorphism group $\Aut(X/S,\De)$ under the natural group homomorphism $\PsAut(X/S,\De) \to {\rm GL}(N^1(X/S)_\Rr)$, respectively.

\begin{conjecture}[Morrison-Kawamata cone conjecture]\label{conj: KM conj}
Let $(X, \De) \to S$ be a klt Calabi-Yau fiber space. 
\begin{enumerate}
\item The cone $\Mov(X/S)_+$ has a (weak) rational polyhedral fundamental domain under the action of $\Gamma_B$, and there are finitely many minimal models of $(X/S, \De)$ up to isomorphisms.
\item The cone $\Amp(X/S)_+$ has a  (weak) rational polyhedral fundamental domain under the action of $\Gamma_A$, and there are finitely many contractions from $X/S$ up to isomorphisms.
\end{enumerate}
\end{conjecture}

Note that there are various versions of the Morrison-Kawamata cone conjecture  (see \cite{LOP18}), and the above version is most relevant to this paper. Besides, we only study the first part of the conjecture in this paper as it is the part that concerns the birational geometry of Calabi-Yau fibrations. For simplicity, we refer to the first part of this conjecture as the (weak) cone conjecture for movable cones, or simply the cone conjecture.

In our recent work \cite{LZ22}, we approach the cone conjecture from the perspective of Shokurov polytopes. We propose a more tractable conjecture and investigate its relationship with the cone conjecture. 

\begin{conjecture}[{\cite{LZ22}}]\label{conj: shokurov polytope}
Let $f: (X, \De) \to S$ be a klt Calabi-Yau fiber space. 
\begin{enumerate}
\item There exists a polyhedral cone $P_M \subset \Eff(X/S)$ such that
\[
\PsAut(X/S, \De)\cdot P_M \supset  \Mov(X/S).
\]
\item There exists a polyhedral cone  $P_A \subset \Eff(X/S)$ such that 
\[
\Aut(X/S, \De)\cdot P_A \supset  \Amp(X/S).
\]
\end{enumerate}
\end{conjecture}

Using results of \cite{Loo14} and assuming standard conjectures of the minimal model program (MMP), \cite{LZ22} showed that Conjecture \ref{conj: shokurov polytope} is almost equivalent to the cone conjecture when $R^1f_*\Oo_X=0$ (i.e., trivial relative irregularity). In fact, if $S$ is a point, then they are indeed equivalent. Furthermore, assuming trivial relative irregularity, we establish the weak Morrison-Kawamata cone conjecture and the finiteness of minimal models, assuming the abundance conjecture and the cone conjecture for the generic fiber.

However, from the Beauville-Bogomolov decomposition of varieties with trivial canonical divisors, varieties with trivial irregularities and non-trivial irregularities may have drastic geometry. In fact, people used to reserve Calabi-Yau varieties only for varieties with trivial canonical divisors and $h^i(\Oo_X)=0, 1\leq i \leq \dim X$. Nonetheless, such an exclusion is neither reasonable from a birational geometry perspective nor practical from a technical standpoint (the cone conjecture is sensitive to finite covers). This paper tackles the general case without assuming trivial relative irregularity. Our starting point is the following result.

\begin{theorem}\label{thm: main 1}
Let $f: (X,\De) \to S$ be a terminal Calabi-Yau fiber space. Assume that good minimal models exist for effective klt pairs in dimension $\dim(X/S)$. If there exists a polyhedral cone $P_\eta \subset \Eff(X_\eta)$ such that 
\[
\PsAut(X_\eta, \De_\eta)\cdot P_\eta \supset  \Mov(X_\eta),
\] then there exists a rational polyhedral cone $Q\subset \Mov(X/S)$ such that
\[
\PsAut(X,\De) \cdot (Q \cap N^1(X/S)_\Qq)= \Mov(X/S)_\Qq.
\]
\end{theorem}

In combination with techniques related to the Shokurov polytope, we deduce the following result:

\begin{theorem}\label{thm: main 2}
Let $f: (X,\De) \to S$ be a terminal Calabi-Yau fiber space. Assume that good minimal models exist for effective klt pairs in dimension $\dim(X/S)$. If there exists a rational polyhedral cone $P_\eta \subset \Eff(X_\eta)$ such that 
\[
\PsAut(X_\eta, \De_\eta)\cdot P_\eta \supset  \Mov(X_\eta),
\] then $(X/S,\De)$ has finite minimal models up to isomorphisms.
\end{theorem}

In the proof of Theorem \ref{thm: main 2}, we need to treat minimal models that are not birational to $X$ in codimension $1$. To this end, we study the finiteness of contractions from $X$. Surprisingly, we establish that the finiteness of (targets of) the contractions can be derived from the movable cone conjecture. This reinforces the fundamental nature of Conjecture \ref{conj: shokurov polytope}, which encompasses both the finiteness of models and contractions.

\begin{theorem}\label{thm: finite contractions}
Let $f: (X,\De) \to S$ be a klt Calabi-Yau fiber space. Assume that good minimal models exist for effective klt pairs in dimension $\dim(X/S)$. If there exists a rational polyhedral cone $Q \subset \Eff(X/S)$ such that 
\[
\PsAut(X/S, \De) \cdot Q \supset \Mov(X/S),
\] then the set $\{Z \mid X \to Z/S \text{~is a contraction}\}$ is finite.
\end{theorem}

\begin{corollary}\label{cor: finite contractions generic case}
Let $f: (X,\De) \to S$ be a terminal Calabi-Yau fiber space. Assume that good minimal models exist for effective klt pairs in dimension $\dim(X/S)$. If there exists a polyhedral cone $P_\eta \subset \Eff(X_\eta)$ such that 
\[
\PsAut(X_\eta)\cdot P_\eta \supset  \Mov(X_\eta),
\] then the set $\{Z \mid X \to Z/S \text{~is a contraction}\}$ is finite.
\end{corollary}

As previously mentioned, the cone conjecture is essential in the study of the birational geometry of varieties with intermediate Kodaira dimensions $\kappa(X)$. As an illustrative example, we prove that $X$ only admits finitely many minimal models given $\dim X - \kappa(X) \leq 2$. It is noteworthy that such an $X$ may no longer be Calabi-Yau varieties.

\begin{corollary}\label{cor: coKodaira leq 2}
Let $X$ be a normal projective variety with canonical singularities such that $\dim X - \kappa(X) \leq 2$. Then $X$ has finitely many minimal models.
\end{corollary}

Another consequence of Theorem \ref{thm: main 1} is the existence of a weak rational polyhedral fundamental domain under the action of the pseudo-automorphism group.

\begin{theorem}\label{thm: weak fundamental domains}
Let $f: (X,\De) \to S$ be a terminal Calabi-Yau fiber space. Assume that good minimal models exist for effective klt pairs in dimension $\dim(X/S)$. If there exists a polyhedral cone $P_\eta \subset \Eff(X_\eta)$ such that 
\[
\PsAut(X_\eta, \De_\eta)\cdot P_\eta \supset  \Mov(X_\eta),
\] then $\Mov(X/S)_+$ admits a weak rational polyhedral fundamental domain under the action of $\Gamma_B$.
\end{theorem}

Applying the above results to the fibrations with $\dim(X/S) \leq 2$, we obtain the following:

\begin{corollary}[{=Corollary \ref{cor: surface fibration}+Corollary \ref{cor: surface fibration fundamental domain}}]
Let $f: X \to S$ be a terminal Calabi-Yau fiber space. Suppose that $\dim(X/S) \leq 2$, then 
\begin{enumerate}
\item $X/S$ has finitely many minimal models, and 
\item $\Mov(X/S)_+$ admits a weak rational polyhedral fundamental domain under the action of $\Gamma_B$.
\end{enumerate}
\end{corollary}

This result generalizes both \cite{Kaw97} and \cite{FHS21} where the finiteness of minimal models are established for threefolds with $\dim(X/S) \leq 2$ and for elliptic fibrations, respectively. 

To elucidate the main idea behind the proof of Theorem \ref{thm: main 1}, we will provide a simplified overview. Let us assume that $X \to S$ is a Calabi-Yau fibration, and denote its generic fiber as $X_\eta$. Suppose that $X_\eta$ satisfies the cone conjecture. In particular, there is a rational polyhedral cone $P_\eta \subset \Eff(X_\eta)$ such that $\PsAut(X_\eta)\cdot P_\eta \supset \Mov(X_\eta)$. Our goal is to lift $P_\eta$ to a rational polyhedral cone $P \subset \Eff(X/S)$ such that $\PsAut(X/S) \cdot P \supset \Mov(X/S)$. This is essential to meet the requirements of Conjecture \ref{conj: shokurov polytope}. If $R^1f_*\Oo_X=0$, then for Cartier divisors $D, B$ on $X$ such that $D_\eta \equiv B_\eta$, we have $D_\eta \sim_\Qq B_\eta$ and thus $D \equiv B/S$ modulo a vertical divisor. However, without $R^1f_*\Oo_X=0$, it may happen that $D \not\equiv B/S$ even after modulo vertical divisors. In other words, there may exist two numerically different divisors (even modulo vertical divisors) on $X$ that look the same in $N^1(X_\eta)$. The idea to overcome this difficulty, as had already appeared in \cite{Kaw97} in special cases, is to use the natural automorphism group $\Aut^0(X_\eta)$ to ensure that, as long as $D_\eta \equiv B_\eta$, 
\[
D \in \Aut^0(X_\eta) \cdot P_B \mod {(\text{vertical divisors})}
\] where $P_B \subset N^1(X/S)$ is a rational polyhedral cone associated with $B$. This process needs to be ``globalized" across all divisors on $P_\eta$ uniformly, resulting in an enlarged rational polyhedral cone that fulfills our objective. 

Nonetheless, overcoming these challenges is not straightforward. 

(1) We are unable to show $\PsAut(X/S) \cdot P \supset \Mov(X/S)$; instead, a weaker statement is sufficient for our purposes, namely that $\PsAut(X/S) \cdot P$ contains the rational elements of $\Mov(X/S)$. This adaptation requires modifications to results concerning the geometry of convex cones, as found in \cite{Loo14} and \cite{LZ22}. 

(2) We need to vastly generalize the aforementioned construction that leverages the numerical equivalence to linear equivalence through the action of $\Aut^0(X_\eta)$. This crucial step draws upon profound results on algebraic groups associated with Calabi-Yau fiber spaces.

We discuss the contents of this paper. Section \ref{sec: preliminaries} furnishes essential background materials and introduces the notation employed throughout the paper. Section \ref{sec: Geometry of convex cones} develops the convex geometry needed in this paper. The main result is Proposition \ref{prop: defined over Q} which is of independent interest. Section \ref{sec: homogeneous fibrations} studies the group schemes associated with Calabi-Yau fibrations. The main result is Theorem \ref{thm: num to linear}. The entirety of Section \ref{sec: fibration vs generic fiber} is devoted to the proof of Theorem \ref{thm: main 1}, the core of this paper. Section \ref{sec: applications} employs the established theorems to prove the remaining results, including those on the finiteness of contractions and minimal models, and the existence of weak fundamental domains.

\subsection*{Acknowledgments} The author is grateful for enlightening discussions with Yong Hu, Jinsong Xu, and Yifei Zhu. We also recognize the substantial influence of \cite{Kaw97} on our work, as it served as our starting point and source of inspiration. It is worth noting that while Shokurov’s polytope technique is not explicitly stated in \cite{Kaw97}, it seems that certain aspects of the proof in \cite{Kaw97} bear a close resemblance to the proof of Shokurov’s polytope.

The author acknowledges partial support from grants provided by the Shenzhen municipality and the Southern University of Science and Technology.

\section{Preliminaries}\label{sec: preliminaries}

Let $f: X \to S$ be a projective morphism between normal quasi-projective varieties over an algebraically closed field of characteristic $0$. Then $f$ is called a fibration if $f$ has connected fibers. We write $X/S$ to mean that $X$ is over $S$. 

For $\mathbb K=\Zz, \Qq, \Rr$ and two $\mathbb K$-Weil divisors $A, B$ on $X$, $A \sim_{\mathbb K} B/S$ means that $A$ and $B$ are $\mathbb K$-linearly equivalent over $S$. We also refer to a $\mathbb K$-divisor as simply a divisor. If $A, B$ are $\Rr$-Cartier divisors, then $A \equiv B/S$ means that $A$ and $B$ are numerically equivalent over $S$. 

We use $\Supp E$ to denote the support of the divisor $E$. A divisor $E$ on $X$ is called a vertical divisor (over $S$) if $f(\Supp E) \neq S$. A vertical divisor $E$ is called a very exceptional divisor if for any prime divisor $P$ on $S$, over the generic point of $P$, we have $\Supp f^*P \not\subset \Supp E$ (see \cite[Definition 3.1]{Bir12b}). If $f$ is a birational morphism, then the notion of a very exceptional divisor coincides with that of an exceptional divisor. 

Let $X$ be a normal complex variety and $\De$ be an $\Rr$-divisor on $X$. Then $(X, \De)$ is called a log pair. We assume that $K_X+\De$ is $\Rr$-Cartier for a log pair $(X, \De)$. 

We call $f: (X, \De) \to S$ a Calabi-Yau fibration or a Calabi-Yau fiber space if $X$ is $\Qq$-factorial, $X \to S$ is a fibration, and $K_X+\De \sim_\Rr 0/S$. When $(X, \De)$ has lc singularities (see Section \ref{subsec: minimal models}),  then $K_X+\De \sim_\Rr 0/S$ is equivalent to the weaker condition  $K_X+\De \equiv 0/S$ by \cite[Corollary 1.4]{HX16}. Note that for an effective $\Rr$-divisor $\De$ such that $(X/S,\De)$ is klt and satisfies $K_X+\De \sim_\Rr 0/S$, then there exists a $\Qq$-divisor $\De'$ such that $(X/S,\De')$ is klt with $K_X+\De'\sim_\Qq 0/S$ and $\Supp \De= \Supp \De'$. Therefore, in terms of the cone conjecture, there is no difference to state it for $\Qq$-linear equivalent or $\Rr$-linear equivalent to $0$.

\subsection{Cones}\label{subsec: cones}

Let $V$ be a finite-dimensional real vector space. A set $C\subset V$ is called a cone if for any $x\in C$ and $\lambda\in \Rr_{>0}$, we have $\lambda \cdot x \in C$. We use $\Int(C)$ to denote the relative interior of $C$ and call  $\Int(C)$ the open cone. By convention, the origin is an open cone.  A cone is called a polyhedral cone (resp. rational polyhedral cone) if it is a closed convex cone generated by finite vectors (resp. rational vectors). If $S \subset V$ is a subset, then $\Conv(S)$ denotes the convex hull of $S$, and $\Cone(S)$ denotes the closed convex cone generated by $S$. As we are only concerned about convex cones in this paper, we also call them cones.

Let $\Pic(X/S)$ be the relative Picard group. Let $$N^1(X/S)_\Zz \coloneqq \Pic(X/S)/{\equiv}$$ be the lattice. Set $\Pic(X/S)_{\mathbb K} \coloneqq \Pic(X/S) \otimes_\Zz \mathbb K$ and $N^1(X/S)_{\mathbb K} \coloneqq N^1(X/S)_\Zz \otimes_\Zz \mathbb K$ for $\mathbb K = \Qq$ or $\Rr$. If $D$ is an $\Rr$-Cartier divisor, then $[D] \in N^1(X/S)_\Rr$ denotes the corresponding divisor class. To abuse the terminology, we also call $[D]$ an $\Rr$-Cartier divisor.

A $\Qq$-Cartier divisor $D$ on $X/S$ is called movable$/S$ if there exists $m\in \Zz_{>0}$ such that the base locus of the linear system $|mD/S|$ has codimension $>1$. An $\Rr$-Cartier divisor is called movable if it is a positive $\Rr$-linear combination of movable Cartier divisors.

We list relevant cones inside $N^1(X/S)_\Rr$ which appear in the paper (the notation is slightly different from \cite{LZ22}):
\begin{enumerate}
\item $\Eff(X/S)$: the cone generated by effective Cartier divisors;
\item $\bEff(X/S)$: the closure of $\Eff(X/S)$;
\item $\Mov(X/S)$: the cone generated by movable divisors;
\item $\bMov(X/S)$: the closure of $\Mov(X/S)$;
\item ${\Mov(X/S)_+}\coloneqq \Conv(\bMov(X/S) \cap N^1(X/S)_\Qq)$ (see Definition \ref{def: polyhedral type});
\item $\Amp(X/S)$: the cone generated by ample divisors;
\item $\bAmp(X/S)$: the closure of $\Amp(X/S)$ (i.e., the nef cone);
\end{enumerate}
To denote the rational elements in the corresponding cones, we set
\begin{enumerate}
\setcounter{enumi}{7}
\item $\Eff(X/S)_\Qq \coloneqq \Eff(X/S)\cap N^1(X/S)_\Qq$;
\item $\Mov(X/S)_\Qq \coloneqq \Mov(X/S)\cap N^1(X/S)_\Qq$;
\item $\Amp(X/S)_\Qq \coloneqq \Amp(X/S)\cap N^1(X/S)_\Qq$;
\end{enumerate}
Moreover, for simplicity, we set
\begin{enumerate}
\setcounter{enumi}{10}
\item $N^1(X/S) \coloneqq N^1(X/S)_\Rr$.
\end{enumerate}

Let $K(S)$ be the field of rational functions of $S$, and $\overline{K(S)}$ be the algebraic closure of $K(S)$. Set $\eta\coloneqq\spec K(S)$ and $\bar\eta \coloneqq \spec \overline{K(S)}$. Then $X_\eta$ and $X_{\bar\eta}$ denote the generic fiber and the geometric fiber of $X \to S$, respectively. The above cones still make sense for $X_\eta/\eta$ and $X_{\bar\eta}/\bar\eta$.

Recall that for a birational map $g: X \dto Y/S$, if $D$ is an $\Rr$-Cartier divisor on $X$, then the pushforward of $D$, denoted as $g_*D$, is defined as follows.  Let $p: W \to X, q: W \to X$ be birational morphisms such that $g \circ p=q$, then $g_*D \coloneqq q_*(p^*D)$. This is independent of the choice of $p$ and $q$. 

 Let $\De$ be a divisor on a $\Qq$-factorial variety $X$. We use $\Bir(X/S, \De)$ to denote the birational automorphism group of $(X, \De)$ over $S$. Specifically, $\Bir(X/S, \De)$ consists of birational maps $ g: X \dto X/S$ such that $g_* \Supp \De= \Supp \De$. A birational map is called a pseudo-automorphism if it is isomorphic in codimension $1$. Let $\PsAut(X/S, \De)$ be the subgroup of  $\Bir(X/S, \De)$ consisting of pseudo-automorphisms. Let $\Aut(X/S, \De)$ be the subgroup of $\Bir(X/S, \De)$ consisting of automorphisms of $X/S$. For a field $K$, if $X$ is a variety over $K$ and $\De$ is a divisor on $X$, then we will omit $\spec K$, and use $\Bir(X, \De), \PsAut(X, \De)$ and $\Aut(X, \De)$ to denote the birational automorphism group, the pseudo-automorphism group and the automorphism group of $X/K$, respectively. See \cite[Example 2.1]{LZ22} for an example of a birational map which is not a pseudo-automorphism.  On the other hand, it is well-known that if $(X/S, \De)$ has terminal singularities and $K_X+\De$ is nef$/S$, then any birational map is a pseudo-automorphism (see Lemma \ref{le: bir=pseudoauto}).

Let $g\in \Bir(X/S, \De)$ and $D$ be an $\Rr$-Cartier divisor on a $\Qq$-factorial variety $X$. Because the pushforward map $g_*$ preserves numerical equivalence classes, there is a linear map (recall that $N^1(X/S)$ denotes $N^1(X/S)_\Rr$ under our convention)
\[
g_*: N^1(X/S) \to N^1(X/S), \quad [D] \mapsto [g_*D].
\] However, if $g\in \Bir(X/S)$ is not isomorphic in codimension $1$, then for a $[D] \in \Mov(X/S)$, $[g_*D]$ may not be in $\Mov(X/S)$. Moreover, $(g , [D]) \mapsto [g_*D]$ is not a group action of $\Bir(X/S, \De)$ on $N^1(X/S)$. For instance, if $D$ is a divisor contracted by $g$, then $g^{-1}_*(g_*[D])= 0 \neq (g^{-1}\circ g)_*[D]$.

On the other hand, it can be verified that
\[
\begin{split}
\PsAut(X/S, \De) \times N^1(X/S) &\to N^1(X/S)\\
(g, [D])& \mapsto [g_*D],
\end{split}
\] is a group action. Note that if $g\in \Aut(X/S,\De)$, then $g_*D = (g^{-1})^*D$. Although for any $g\in \PsAut(X/S)$, we can still define the pullback map on an $\Rr$-Cartier divisor, in order to make $\PsAut(X/S,\De)$ acting on $N^1(X/S)_\Rr$ from the left, we use the pushforward map.

We use $g \cdot D$ and $g \cdot [D]$ to denote $g_*D$ and $[g_*D]$, respectively. Let $\Gamma_B$ and $\Gamma_A$ be the images of $\PsAut(X/S, \De)$ and $\Aut(X/S, \De)$ under the natural group homomorphism $$\iota: \PsAut(X/S, \De) \to {\rm GL}(N^1(X/S)).$$ Because $\Gamma_B, \Gamma_A \subset {\rm GL}(N^1(X/S)_\Zz)$, $\Gamma_B$ and $\Gamma_A$ are discrete subgroups. By abusing the notation, we also write $g$ for $\iota(g) \in \Gamma_B$, and denote $\iota(g)([D])$ by $g\cdot [D]$. Then the cones $\Mov(X/S)_\Qq, \bMov(X/S)$ and $\Mov(X/S)_+$ are all invariant under the action of $\PsAut(X/S,\De)$. Similarly, $\Amp(X/S)_\Qq, \Amp(X/S)$ and $\Amp(X/S)_+$ are all invariant under the action of $\Aut(X/S,\De)$.

We use $\bAut(X/S,\De)$ and $\bPic(X/S)$ to denote the group schemes that represent the automorphism functor and the Picard functor, respectively. See Section \ref{sec: fibration vs generic fiber} for details.

\subsection{Minimal models of varieties}\label{subsec: minimal models}

Let $(X,\De)$ be a log pair. For a divisor $D$ over $X$, if $f: Y \to X$ is a birational morphism from a smooth variety $Y$ such that $D$ is a divisor on $Y$, then the log discrepancy of $D$ with respect to $(X, \De)$ is defined to be $$a(D; X, \De) \coloneqq\mult_{D}(K_Y-f^*(K_X+\De))+1.$$ This definition is independent of the choice of $Y$. A log pair $(X, \De)$ (or its singularity) is called sub-klt (resp. sub-lc) if the log discrepancy of any divisor over $X$ is $>0$ (resp. $\geq 0$). If $\De \geq 0$, then a sub-klt (resp. sub-lc) pair $(X, \De)$ is called klt (resp. lc). If $\De \geq 0$ and the log discrepancy of any exceptional divisor over $X$ is $>1$, then $(X,\De)$ is said to have terminal singularities. A fibration/fiber space $(X, \De) \to S$ is called a klt (resp. terminal) fibration/fiber space if $(X, \De)$ is klt (resp. terminal).

Let $X \to S$ be a projective morphism of normal quasi-projective varieties. Suppose that $(X, \De)$ is klt. Let $\phi: X \dto Y/S$ be a birational contraction (i.e., $\phi$ does not extract divisors) of normal quasi-projective varieties over $S$, where $Y$ is projective over $S$. We write $\De_Y \coloneqq \phi_*\De$ for the strict transform of $\De$. Then $(Y/S, \De_Y)$ is a minimal model of $(X/S, \De)$ if $K_Y+\De_Y$ is nef$/S$ and $a(D; X, \De) \geq a(D;Y, \De_Y)$ for any divisor $D$ over $X$. Note that a minimal model here is called a weak log canonical model in \cite{BCHM10}.

A minimal model $(Y/S, \De_Y)$ of $(X/S, \De)$ is called a good minimal model of $(X/S, \De)$ if $K_Y+\De_Y$ is semi-ample$/S$. It is well-known that the existence of a good minimal model of $(X/S, \De)$ implies that any minimal model of $(X/S, \De)$ is a good minimal model (for example, see \cite[Remark 2.7]{Bir12b}).

By saying that ``good minimal models of effective klt pairs exist in dimension $n$", we mean that for any projective variety $X$ of dimension $n$ over an algebraically closed field of characteristic $0$, if $(X, \De)$ is klt and the Kodaira dimension $\ka(K_X+\De) \geq 0$, then $(X, \De)$ has a good minimal model. 

The following lemma allows to lift a minimal model so that it is isomorphic to the original variety in codimension $1$.

\begin{lemma}[{\cite[Lemma 2.2]{LZ22}}]\label{le: lift to iso in codim 1}
Let $(X/S, \De)$ be a klt pair with $[K_X+\De] \in \bMov(X/S)$. Suppose that $g: (X/S, \De) \dto (Y/S, \De_Y)$ is a minimal model of $(X/S, \De)$. Then $(X/S, \De)$ admits a minimal model $(Y'/S, \De_{Y'})$ such that 
\begin{enumerate}
\item $Y'$ is $\Qq$-factorial,
\item $X, Y'$ are isomorphic in codimension $1$, and
\item there exists a morphism $\nu: Y' \to Y/S$ such that $K_{Y'}+\De_{Y'} = \nu^*(K_Y+\De_Y)$.
\end{enumerate}
\end{lemma}

\begin{theorem}[{\cite[Theorem 2.12]{HX13}}]\label{thm: HX13}
Let $f: X \to S$ be a surjective projective morphism and $(X, \De)$ a klt pair such that for a very general closed point $s\in S$, the fiber $(X_s, \De_s=\De|_{X_s})$ has a good minimal model. Then $(X, \De)$ has a good minimal model over $S$.
\end{theorem}

\cite[Theorem 2.12]{HX13} states for a $\Qq$-divisor $\De$. However, it still holds for an $\Rr$-divisor $\De$: in the proof of \cite[Theorem 2.12]{HX13}, one only needs to replace $\proj_S \oplus_{m\in\Zz_{>0}}R^0f_*\Oo_X(m(K_X+\De))$ by the canonical model of $(X/S, \De)$ whose existence is known for effective klt pairs by \cite{Li22}. Indeed, because $\kappa(K_{X_s}+\De_s) \geq 0$ for a very general $s\in S$ by assumption, $K_X+\De \sim_\Rr E/S$ with $E \geq 0$ by \cite[Theorem 3.15]{Li22}.

Let $V$ be a finite-dimensional $\Rr$-vector space. A polytope (resp. rational polytope) $P\subset V$ is the convex hull of finite points (resp. rational points) in $V$. In particular, a polytope is always closed and bounded. We use $\Int(P)$ to denote the relative interior of $P$ and call $\Int(P)$ the open polytope. By convention, a single point is an open polytope.  Therefore, $\Rr_{>0}\cdot P$ is an open polyhedral cone iff $P$ is an open polytope. See \cite{LZ22} for the proof of the following results:

\begin{theorem}[{\cite[Theorem 3.4]{SC11}}]\label{thm: Shokurov}
Let $X$ be a $\Qq$-factorial variety and $f: X \to S$ be a  fibration. Assume that good minimal models exist for effective klt pairs in dimension $\dim(X/S)$. Let $D_i, i=1,\ldots, k$ be effective $\Qq$-divisors on $X$. Suppose that $P \subset \oplus_{i=1}^k [0,1) D_i$ is a rational polytope such that for any $\De \in P$, $(X, \De)$ is klt and $\ka(K_F+ \De|_F)\geq 0$, where $F$ is a general fiber of $f$. 

Then $P$ can be decomposed into a disjoint union of finitely many open rational polytopes $P = \sqcup_{i=1}^m Q^\circ_i$ such that for any $B, D \in Q^\circ_i$, if $(Y/S, B_Y)$ is a minimal model of $(X/S, B)$, then $(Y/S, D_Y)$ is also a minimal model of $(X/S, D)$.
\end{theorem}

\begin{theorem}[{\cite[Theorem 2.6]{LZ22}}]\label{thm: Shokurov-Choi}
Let $(X,\De)\to S$ be a klt Calabi-Yau fiber space. Assume that good minimal models of effective klt pairs exist in dimension $\dim(X/S)$. Let $P \subset \Eff(X/S)_\Qq$ be a rational polyhedral cone. Then $P$ is a finite union of open rational polyhedral cones $P = \sqcup_{i=0}^m P^\circ_i$ such that whenever
\begin{enumerate}
\item $B, D$ are effective divisors with $[B], [D] \in P^\circ_i$, and
\item $(X, \De+\ep B), (X, \De+\ep D)$ are klt for some $\ep \in \Rr_{> 0}$,
\end{enumerate} 
then if $(Y/S, \De_Y+\ep B_Y)$ is a minimal model of $(X/S, \De+\ep B)$, then $(Y/S, \De_Y+\ep D_Y)$ is a minimal model of $(X/S, \De+\ep D)$.
\end{theorem}

We need the following consequence of Theorem  \ref{thm: Shokurov-Choi}.

\begin{lemma}\label{le: inclusion implies equal}
Let $f: (X,\De) \to S$ be a klt Calabi-Yau fiber space. Assume that good minimal models exist for effective klt pairs in dimension $\dim(X/S)$. Let $Q \subset \Eff(X/S)$ be a rational polyhedral cone, and $\Gamma \subset \PsAut(X/S, \De)$ be a subgroup. Then there exists a rational polyhedral subcone $P \subset Q \cap \Mov(X/S)$ such that 
\begin{enumerate}
\item $\Gamma \cdot (P \cap N^1(X/S)_\Qq) = (\Gamma \cdot Q)\cap \Mov(X/S)_\Qq$, and
\item $\Gamma \cdot P = (\Gamma \cdot Q)\cap \Mov(X/S)$.
\end{enumerate}
\end{lemma}
\begin{proof}
We only show the first item as the second item can be proved by the same argument. The argument below is similar to the proof of \cite[Lemma 5.2]{LZ22}.

Let $Q = \sqcup_{i=1}^m Q_i^\circ$ be the decomposition as in Theorem \ref{thm: Shokurov-Choi} such that for effective divisors $B, D$ with $[B], [D] \in Q_i^\circ$ and $1 \gg \ep >0$, the klt pairs $(X, \De+\ep B)$ and $(X, \De+\ep D)$ share same minimal models. We claim that if $Q^\circ_i \cap \Mov(X/S)_\Qq \neq \emptyset$, then $(Q_i \cap N^1(X/S)_\Qq)\subset \Mov(X/S)_\Qq$. Let $h: X \dto Y/S$ be a minimal model of $(X, \De+\ep D)$ for some $[D] \in Q^\circ_i \cap \Mov(X/S)_\Qq \neq \emptyset$ and $1\gg \ep>0$. By Lemma \ref{le: lift to iso in codim 1}, we can assume that $h$ is isomorphic in codimension 1. Then for any $B \geq 0$ such that $[B] \in Q_i \cap N^1(X/S)_\Qq$, $h$ is also a minimal model for $(X, \De+\ep' B)$ with $1\gg\ep'>0$. By Theorem \ref{thm: HX13}, $K_Y+\De_Y+\ep'B'_Y \sim_\Qq \ep' B_Y/S$ is semi-ample over $S$. Thus, $B$, as the strict transform of $B_Y$, is movable over $S$.  

Set 
\[
P \coloneqq \Cone(\cup_i Q_i \mid (Q^\circ_i \cap \Mov(X/S)_\Qq \neq \emptyset) \subset Q.
\] Then $P$ is a rational polyhedral cone and $P \subset \Mov(X/S)$ by the previous claim. If $[D] \in (\Gamma \cdot Q) \cap \Mov(X/S)_\Qq$, then there exists $\gamma \in \Gamma$ such that $\gamma\cdot [D] \in Q$ and thus lies in $Q^\circ_i$ for some $i$. In particular, $[D] \in \gamma^{-1}\cdot P$. This shows $\Gamma\cdot (P\cap N^1(X/S)_\Qq) \supset (\Gamma \cdot Q) \cap\Mov(X/S)_\Qq$. The inverse inclusion follows from $P \subset Q$.
\end{proof}

The following result is well-known.

\begin{lemma}\label{le: bir=pseudoauto}
If $(X/S, \De)$ has terminal singularities and $K_X+\De$ is nef$/S$, then $\Bir(X) = \PsAut(X)$.
\end{lemma}
\begin{proof}
Replacing $(X/S, \De)$ by a small $\Qq$-factorial modification, we can assume that $X$ is $\Qq$-factorial. Let $g: X \dto X$ be birational and $p, q: W \to X$ be birational resolutions such that $g=q \circ p^{-1}$. Let $\De_W$ be the strict transform of $\De$. Let $\Exc(-)$ denote the exceptional locus of a birational morphism. As $(X/S, \De)$ has $\Qq$-factorial terminal singularities, $K_W+\De_W=p^*(K_X+\De) +E$ and  $K_W+\De_W=q^*(K_X+\De) +F$ with $E, F \geq 0$ and $\Supp E =\Exc(p), \Supp F =\Exc(q)$. Thus $p^*(K_X+\De) =q^*(K_X+\De)$ by the negativity lemma, and $\Exc(p)=\Exc(q)$. Hence 
 \[
X \backslash p(\Exc(p))\stackrel{p^{-1}}{\simeq} W \backslash \Exc(p)=W \backslash \Exc(q) \stackrel{q}{\simeq} X \backslash q(\Exc(q)),
 \] and $g$ is isomorphic in codimension 1.
 \end{proof}

\section{Geometry of convex cones}\label{sec: Geometry of convex cones}

Let $V(\Zz)$ be a lattice and $V(\Qq) \coloneqq V(\Zz) \otimes_\Zz \Qq$, $V  \coloneqq V(\Qq) \otimes_\Qq \Rr$. A cone $C \subset V$ is non-degenerate if it does not contain an affine line. This is equivalent to saying that its closure $\bar C$ does not contain a non-trivial vector space.

In the following, we assume that $\Gamma$ is a group and $\rho: \Gamma \to {\rm GL(V)}$ is a group homomorphism. The group $\Gamma$ acts on $V$ through $\rho$. For a $\gamma \in \Gamma$ and an $x \in V$, we write $\gamma \cdot x$ or $\gamma x$ for the action. For a set $S\subset V$, set $\Gamma \cdot S \coloneqq \{\gamma \cdot x \mid  \gamma \in \Gamma, x\in S\}$. Suppose that this action leaves a convex cone $C$ and some lattice in $V(\Qq)$ invariant. We assume that $\dim C=\dim V$. The following definition slightly generalizes \cite[
Proposition-Definition 4.1]{Loo14}.

\begin{definition}\label{def: polyhedral type} 
Under the above notation and assumptions.
\begin{enumerate}
\item Suppose that $C \subset V$ is an open convex cone (may be degenerate). Let
\[
C_+ \coloneqq \Conv(\bar C \cap V(\Qq))
\] be the convex hull of rational points in $\bar C$. 

\item We say that $(C_+, \Gamma)$ is of polyhedral type if there is a polyhedral cone $\Pi \subset C_+$ such that $\Gamma \cdot \Pi \supset C$.
\end{enumerate}
\end{definition}

For a set $S \subset V$, $[S]$ denotes the convex hull of $S$ in $V$. A point of a convex set that makes up a face by itself is also called an extreme point.

\begin{proposition}[{\cite[Proposition-Definition 4.1]{Loo14}}]\label{prop: prop-def}
Under the above notation and assumptions. If $C$ is an open non-degenerate cone, then the following conditions are equivalent:
\begin{enumerate}
\item there exists a polyhedral cone $\Pi \subset C_+$ with $\Gamma \cdot  \Pi = C_+$;
\item there exists a polyhedral cone $\Pi \subset C_+$ with $\Gamma \cdot  \Pi \supset C$;
\item there exists a polyhedral cone $\Pi \subset C_+$ with $\Gamma \cdot  \Pi \supset C\cap V(\Qq)$;
\item For every $\Gamma$-invariant lattice $L \subset V(\Qq)$, $\Gamma$ has finitely many orbits in the set of extreme points of $[C \cap L]$.
\end{enumerate}
Moreover, in case (2), we necessarily have $\Gamma \cdot  \Pi = C_+$.
\end{proposition}

\begin{proof}
\cite[Proposition-Definition 4.1]{Loo14} has shown that (1) (2) (4) are equivalent. It is straightforward to get (3) from (2). Hence it suffices to show that (3) implies (4). This can be shown by the same argument as \cite[Proposition-Definition 4.1]{Loo14} in showing (2) implies (4). We copy the argument below (with minor modifications and some explanations) for the convenience of the reader.

We prove that (3) implies (4). Without loss of generality, we may assume that $\Pi$ is rationally polyhedral. Let $S$ denote the set of extreme points of $[C \cap L]$. In Step 1 of the proof of \cite[Theorem 2.2]{Loo14}, it is shown that $[C \cap L] = [C \cap L]+\bar C$. Then \cite[Lemma 1.6]{Loo14} implies that every extreme point of $[C \cap L]$ belongs to $C \cap L \subset V(\Qq)$. 

As $S$ is $\Gamma$-invariant and $\Gamma \cdot \Pi \supset C\cap V(\Qq)\supset S$, to show (4), it suffices to show that $S \cap \Pi$ is finite. Let $v_1,\cdots, v_r$ denote the set of primitive integral generators of the extremal rays of $\Pi$. Any $e \in S \cap \Pi$ has the property that 
\begin{equation}\label{eq: not in}
e-v_i \notin C \cap\Pi
\end{equation} for all $i$. Otherwise, suppose that $e-v_i \in C \cap\Pi\cap L$. By $v_i \in \Pi \subset C_+$ and $e \in C$ with $C$ open, we have $e+v_i \in C \cap L$. Hence 
\[
e = \frac 1 2 (e-v_i)+\frac 1 2 (e+v_i) \in [C \cap L].
\] This contradicts that $e$ is an extreme point of $[C\cap L]$.

Now \eqref{eq: not in} implies that if we write $e = \sum_{i=1}^r \lambda_iv_i$ with $\lambda_i \geq 0$, then $\lambda_i \leq 1$ for all $i$. So $S\cap \Pi$ is contained in a compact set and hence finite.

\end{proof}

\begin{definition}\label{def: fundamental domain}
Let $\rho: \Gamma \hookrightarrow {\rm GL}(V)$ be an injective group homomorphism and $C \subset V$ be a cone (may not necessarily be open). Let $\Pi \subset C $ be a (rational) polyhedral cone. Suppose that $\Gamma$ acts on $C$. Then $\Pi$ is called a weak (rational) polyhedral fundamental domain for $C$ under the action $\Gamma$ if 
\begin{enumerate}
\item $\Gamma \cdot \Pi = C$, and
\item for each $\gamma \in \Gamma$, either $\gamma \Pi = \Pi$ or $\gamma\Pi \cap \Int(\Pi) = \emptyset$.
\end{enumerate}

Moreover, let $\Gamma_\Pi\coloneqq \{\gamma \in \Gamma \mid \gamma \Pi = \Pi\}$. If $\Gamma_\Pi=\{{\rm Id}\}$, then $\Pi$ is called a (rational) polyhedral fundamental domain.
\end{definition}

See \cite[Lemma 3.5]{LZ22} for the following application of {\cite[Theorem 3.8 \& Application 4.14]{Loo14}}.

\begin{lemma}[{\cite[Theorem 3.8 \& Application 4.14]{Loo14}}]\label{le: existence of fun domain}
Under the notation and assumptions of Definition \ref{def: polyhedral type}. Suppose that $\rho: \Gamma \hookrightarrow {\rm GL}(V)$ is injective. Let $(C_+, \Gamma)$ be of polyhedral type with $C$ non-degenerate. Then under the action of $\Gamma$, $C_+$ admits a rational polyhedral fundamental domain.
\end{lemma}

For a possibly degenerate open convex cone $C$, let $W \subset \overline C$ be the maximal $\Rr$-linear vector space. We say that $W$ is defined over $\Qq$ if $W = W(\Qq) \otimes_\Qq \Rr$ where $W(\Qq) = W \cap V(\Qq)$. In this case, $V/W=(V(\Qq)/W(\Qq)) \otimes_\Qq \Rr$ has a nature lattice structure, and we denote everything in $V/W$ by $\ti{(-)}$. For example, $\widetilde{(C_+)}$ is the image of $C_+$ under the projection  $p:V \to V/W$. By the maximality, $W$ is $\Gamma$-invariant, and thus $V/W, \ti C$ admit natural $\Gamma$-actions.

\begin{lemma}[{\cite[Lemma 3.7]{LZ22}}]\label{le: induced polyhedral type}
Under the above notation and assumptions, 
\begin{enumerate}
\item $\overline{\ti C} = \ti{\overline C}$,
\item $({\ti C})_+ = \widetilde{(C_+)}$, which is denoted by $\ti C_+$, and
\item if $(C_+, \Gamma)$ is of polyhedral type, then $(\ti C_+, \Gamma)$ is still of polyhedral type. More precisely, if $\Pi \subset C_+$ is a polyhedral cone with $\Gamma \cdot \Pi \supset C$, then $\ti \Pi \subset \ti C_+$ and $\Gamma \cdot \ti\Pi \supset \ti C$. 
\end{enumerate}
\end{lemma} 

The following result can be shown by the same argument as \cite[Lemma 3.7]{LZ22}.

\begin{proposition}\label{prop: lift cone}
Let $W \subset \overline C$ be the maximal vector space. Suppose that $W$ is defined over $\Qq$. Let $\ti \Gamma$ be the image of the natural group homomorphism $\Gamma \to {\rm GL}(V/W)$. If $(\ti C_+, \ti\Gamma)$ is of polyhedral type, then there is a rational polyhedral cone $\Pi \subset C_+$ such that $\Gamma \cdot \Pi = C_+$, and for each $\gamma \in \Gamma$, either $\gamma \Pi\cap \Int(\Pi) = \emptyset$ or $\gamma \Pi = \Pi$. Moreover, 
\[
\{\gamma \in \Gamma \mid \gamma \Pi = \Pi\} = \{\gamma \in \Gamma \mid \gamma \text{~acts trivially on ~}V/W\}.
\]
\end{proposition}
\begin{proof}
By Lemma \ref{le: existence of fun domain}, there is a rational polyhedral cone $\ti \Pi$ as a fundamental domain of $\ti C_+$ under the action of $\ti\Gamma$. By Lemma \ref{le: induced polyhedral type} (2), let $\Pi' \subset C_+$ be a  rational polyhedral cone such that $p(\Pi')=\ti\Pi$, where $p: V \to V/W$. Let $\Pi \coloneqq  \Pi'+W$ which is a rational polyhedral cone. As $W$ is defined over $\Qq$, we have $W \subset C_+$ and $\Pi \subset C_+$. As $\gamma (\Pi'+W) = (\gamma \Pi')+W$, we have $\Gamma \cdot \Pi = C_+$ by Lemma \ref{le: induced polyhedral type} (2). 

If $\gamma \ti\Pi \cap \Int(\ti\Pi) = \emptyset$, then $\gamma \Pi \cap \Int(\Pi) = \emptyset$ as $\Int(\Pi)$ maps to $\Int(\ti\Pi)$. If $\gamma \ti\Pi =\ti\Pi$, then we claim that $\gamma \Pi = \Pi$. In fact, for some $a\in \Pi'$, we have $\widetilde{(\gamma  \cdot  a)}=\gamma  \cdot \ti a \in \ti \Pi$ and thus $\gamma  \cdot  a = b+w$ for some $b\in \Pi', w\in W$. Thus $\gamma  \Pi \subset \Pi$. Similarly, $\gamma^{-1}  \Pi \subset \Pi$. This shows the claim. Moreover, $\gamma \Pi = \Pi$ iff $\gamma$ acts trivially on $\ti \Pi$ iff $\gamma$ acts trivially on $V/W$ because $\ti \Pi$ is a fundamental domain under the action of $\ti\Gamma$.
\end{proof}

There are natural maps between N\'eron-Severi spaces of fibrations and generic/geometric fibers. 

\begin{lemma}\label{le: natural maps}
Let $X \to S$ be a fiber space, and $U\subset S$ be an open set. Then there exist the following natural maps
\begin{enumerate}
\item $N^1(X/S) \to N^1(X_{\eta})$ with $[D] \mapsto [D_{\eta}]$,
\item $N^1(X/S) \to N^1(X_U/U)$ with $[D] \mapsto [D_U]$,
\item $N^1(X/S) \to N^1(X_{\bar\eta})$ with $[D] \mapsto [D_{\bar\eta}]$, and
\item $N^1(X_\eta) \to N^1(X_{\bar\eta})$ with $[D] \mapsto [D_{\bar\eta}]$.
\end{enumerate}
Moreover, for any sufficiently small $U$, we have $N^1(X_U/U) \simeq N^1(X_\eta)$, and they both map to $N^1(X_{\bar\eta})$ injectively.
\end{lemma}
\begin{proof}
The (2) follows from the definition directly, the remaining items can be proved by a similar method as \cite[Proposition 4.3]{LZ22}. Hence we just sketch the proof of (4) below.

First, as $N^1(X_\eta)$ is defined over $\Qq$, it suffices to show that for a Cartier divisor $D$ on $X_\eta$ such that $D \equiv 0/\eta$, we have $D_{\bar\eta} \equiv 0/\bar\eta$. Let $\ti C$ be a curve on $X_{\bar\eta}$. Then $\ti C$ is defined over a finite extension of $K(S)$. In other words, there exist a field extension $F/K(S)$ with $[F: K(S)]=d<\infty$ and a curve $C$ on $X_F=X \times_\eta \spec F$ such that $C_{\bar\eta} = \ti C$. By the property of flat base extension, we have
\[
\chi(\ti C, mD_{\bar\eta}) = \chi(C, mD_F) \text{~for all~} m\in \Zz,
\] where $D_F = D \times_\eta \spec F$. By the definition of the intersection, we have $D_{\bar\eta} \cdot \ti C = D_F \cdot C$. 
Let $C'$ be the image of $C$ under the natural morphism $X_F \to X_{\eta}$. Then $C'$ is a curve on $X_{\eta}$. By the same proof of the projection formula (see \cite[Proposition 1.10]{Deb01}), we have $D_F \cdot C = d(D\cdot C')$. This shows the claim.

In \cite[Proposition 4.3]{LZ22}, we showed that $N^1(X_U/U) \to N^1(X_{\bar\eta})$ is injective for any sufficiently small $U$. As $N^1(X_U/U) \to N^1(X_{\eta})$ is surjective and $N^1(X_U/U) \to N^1(X_{\bar\eta})$ factors through $N^1(X_U/U) \to N^1(X_{\eta})$, we see that $N^1(X_U/U) \to N^1(X_{\eta})$ is injective. Hence $N^1(X_U/U) \simeq N^1(X_\eta)$, and they both map to $N^1(X_{\bar\eta})$ injectively.
\end{proof}

The following proposition is useful to study Calabi-Yau fibrations with non-trivial $R^1f_*\Oo_X$. It partially answers a question in \cite[Question 4.5]{LZ22}.

\begin{proposition}\label{prop: defined over Q}
Let $(X, \De) \to S$ be a klt Calabi-Yau fiber space. Let $E$ and $M$ be the maximal vector spaces in $\bEff(X/S)$ and $\bMov(X/S)$, respectively. Then $E$ and $M$ are defined over $\Qq$.
\end{proposition}
\begin{proof}
By Lemma \ref{le: natural maps} (3), there exists the natural map 
\[
\theta: N^1(X/S) \to N^1(X_{\bar\eta}), \quad [D] \mapsto [D_{\bar\eta}].
\] We claim that $\Ker (\theta) = E$. Hence, $E$ is defined over $\Qq$. For $[D] \in \Ker (\theta)$, we have $D_{\bar\eta} \equiv 0$ on $X_{\bar\eta}$, hence $D_{\bar\eta}+tA_{\bar\eta}$ is big for any ample$/S$ divisor $A$ on $X$ and $t \in \Rr_{>0}$. Thus $D+tA$ is also big over $S$. Take $t \to 0$, we have $[D] \in \bEff(X/S)$. For the same reason, $[-D] \in \bEff(X/S)$. Hence $[D]\in E$. Conversely, if $[D] \in E$, then $\pm [D_{\bar\eta}] \in \bEff(X_{\bar\eta})$. Elements in $\bEff(X_{\bar\eta})$ intersect with movable curves non-negatively (see \cite{BDPP13}). As movable curves form a full dimensional cone in $N_1(X_{\bar\eta})$, we have $D_{\bar\eta} \equiv 0$.

To show that $M$ is defined over $\Qq$. Suppose that $X' \to S$ is a fiber space such that $X'$ is $\Qq$-factorial and $X \dto X'/S$ is isomorphic in codimension $1$. If $D$ is a divisor on $X$, then let $D'$ be the push-forward of $D$ on $X'$. If $C$ is an irreducible curve on $X'$, then we say that its class covers a divisor if the Zariski closure 
\[
\overline{\bigcup_{[C']=[C] \in N_1(X'/S)} C'}
\] in $X'$  has codimension $\leq 1$. 

We claim that 
\begin{equation}\label{eq: 1}
\begin{split}
M = E \cap &\{[D] \mid D'\cdot C=0, \text{~where~} X \dto X'/S \text{~is isomorphic in codimension~}1, \\
&\qquad X' \text{~is~} \Qq\text{-factorial, and~} C \text{~is a curve whose class covers a divisor}\}.
\end{split}
\end{equation}

For ``$\subset$", if $[D] \in M$, then $\pm [D'] \in \bMov(X'/S)$, and thus $\pm D' \cdot C \geq 0$ for any curve $C$ whose class covers a divisor in $X'$. Hence, $D' \cdot C=0$.

For ``$\supset$", let $[D]$ belong to the right-hand side of \eqref{eq: 1}. As $[D] \in E$ and $K_X+\De \equiv 0/S$, we can run a $D$-MMP with scaling of an ample divisor $A$ over $S$ (e.g. \cite{BCHM10}). If this MMP consists of divisorial contractions or Mori fibrations, then let $X'$ be the first variety where this occurs. Hence $X \dto X'/S$ is isomorphic in codimension $1$ and there is a curve $C$ on $X'$ whose class covers a divisor such that $D'\cdot C<0$. This contradicts the choice of $D$. Hence the MMP only consists of flips. If this MMP terminates with $X'$, then $D'$ is nef$/S$, and thus $[D]\in \bMov(X/S)$. If this MMP does not terminate, then the nef threshold $\lambda_i$ in the scaling must approach $0$ by \cite{BCHM10}, thus $[D]=\lim_{i \to \infty}[D+\lambda_i A] \in \bMov(X/S)$. The same argument shows that $[-D]\in\bMov(X/S)$. Therefore, $[D]\in M$.  As $N^1(X/S) \to N^1(X'/S)$ is a linear map defined over $\Qq$, and the curve class $[C]$ in $N_1(X'/S)$ is an integral element, the vector space
\[
\begin{split}
&\{[D] \mid D'\cdot C=0, \text{~where~} X \dto X'/S \text{~is isomorphic in codimension~}1, \\
&\qquad X' \text{~is~} \Qq\text{-factorial, and~} C \text{~is a curve whose class covers a divisor}\}.
\end{split}
\] is defined over $\Qq$. As $E$ is defined over $\Qq$, $M$ is also defined over $\Qq$.
\end{proof}

\section{Groups schemes associated to log Calabi-Yau varieties}\label{sec: homogeneous fibrations}

The main result of this section is Theorem \ref{thm: num to linear} that rectifies the numerical equivalence to linear equivalence with the help of automorphism groups. 

Let $X$ be a scheme over a scheme $S$ and $\Aut(X/S)$ be the set of automorphism of $X$ over $S$. Let $\CAut_{X/S}$ be the automorphism functor defined by
\[
\CAut_{X/S}(T) \coloneqq \Aut(X_T/T),
\] where $T$ is a scheme over $S$ and $X_T = X \times_S T$. If $S$ is noetherian and $X \to S$ is a flat projective morphism, then $\CAut_{X/S}$ is representable by a group scheme $\bAut(X/S)$ which is locally of finite type over $S$ (e.g. see \cite[\S 5.6.2]{Nit05}). In particular, we have $\bAut(X_\eta/\eta)$ and $\bAut(X_{\bar\eta}/{\bar\eta})$ that represent the automorphism functors on the generic and geometric fibers, respectively. For simplicity, we write $\bAut(X_\eta)$ for $\bAut(X_\eta/\eta)$ and $\bAut(X_{\bar\eta})$ for $\bAut(X_{\bar\eta}/{\bar\eta})$. For a $k$-scheme $V$, and a field $K$ over $k$, define $V_K \coloneqq V \times_{\spec k} \spec K$, and
\[
V(\spec K) \coloneqq \Hom_{\spec K}(\spec K, V_K)
\] as the set of $(\spec K)$-point of $V$. Thus, $\bAut(X_\eta)(\eta)=\Aut(X_\eta)$. We also use $V(K)$ to denote $V(\spec K)$ when there is no ambiguity. We use $[g] \in V_K$ to denote the $(\spec K)$-point of $V_K$ that corresponds to $g\in V(K)$. Thus, $\bAut(X_\eta)(\eta)=\Aut(X_\eta)$. Let $\De$ be a divisor on $X/S$, and $\De_T \coloneqq \De \times_S T$ for a base extension $T\to S$. Define  $\CAut_{X/S,\De}$ as the sub-functor of $\CAut_{X/S}$ such that
\[
\CAut_{X/S,\De}(T) \coloneqq \{g\in \Aut(X_T/T)\} \mid g(\Supp \De_T) = \Supp \De_T\}. 
\] If $S$ is noetherian and $X \to S$ is a flat projective morphism, then $\CAut_{X/S,\De}$ is also representable by a subgroup scheme $\bAut(X/S, \De)$. 

When $S =\spec K$ with $K$ a field, we use $\bAut^0(X/S,\De)$ to denote the connected component of $\bAut(X/S,\De)$ that contains the identity element (see \cite[\S 9.5]{Kle05}), and set
\[
\Aut^0(X/S,\De) \coloneqq \bAut^0(X/S,\De)(K).
\] 

Let $X$ be a scheme over a scheme $S$. Following \cite[Definition 9.2.2]{Kle05}, we define the relative Picard functor $\CPic_{X/S}$ to be 
\[
\CPic_{X/S}(T) \coloneqq \Pic(X_T)/\Pic(T).
\]Denote the associated sheaf in the \'etale topology by
\[
\CPic_{X/S, \text{(\'et)}}.
\] See \cite[\S 2.3.7]{Vis05} for the sheafification of a functor. By \cite[Theorem 9.4.8]{Kle05}, if $X \to S$ is projective Zariski locally over $S$, and is flat with integral geometric fibers, then $\CPic_{X/S, \text{(\'et)}}$ is representable by a separated scheme which is locally of finite type over $S$, denoted by $\bPic(X/S)$. When $S =\spec K$ with $K$ a field, we use $\bPic^0(X/S)$ to denote the connected component of $\bPic(X/S)$ that contains the identity element. When there is no risk of ambiguity, $S$ will be omitted.

In particular, for a fibration $X \to S$, $\CPic_{X_\eta/\eta, \text{(\'et)}}$ and $\CPic_{X_{\bar\eta}/\bar\eta, \text{(\'et)}}$ are representable by $\bPic(X_\eta/\eta)$ and $\bPic(X_{\bar\eta}/\bar\eta)$, respectively. One should notice that although
\[
\Hom_{\bar\eta}(\bar\eta, \bPic(X_{\bar\eta}/\bar\eta))=\CPic_{X_{\bar\eta}/\bar\eta, \text{(\'et)}}(\bar\eta) = \Pic(X_{\bar\eta}),
\] the $\Hom_{\eta}(\eta, \bPic_{X_{\eta}/\eta})=\CPic_{X_{\eta}/\eta, \text{(\'et)}}(\eta)$ may contain more elements than $\Pic(X_\eta)$ due to the sheafification (see \cite[Exercise 9.2.4]{Kle05}).

The results presented in the rest of this section will be applied to the generic fiber of a Calabi-Yau fibration. Hence, from now on until the end of the section, we work with a projective log pair over a field of characteristic $0$ (but may not be algebraically closed).

The following theorem is base on \cite{Bri10, Bri13} (see \cite{Kaw85, Amb16} for similar results).

\begin{theorem}[{\cite{Xu20}}]\label{thm: Xu}
Let $(X, \De)$ be a projective log pair with klt singularities over an algebraically closed field of characteristic $0$. Assume that the log canonical divisor $K_X + \De \sim_\Rr 0$. Then
\begin{enumerate}
\item {\rm (\cite[heorem 4.5]{Xu20})}. The connected component of the group scheme $\bAut(X, \De)$ containing the identity morphism, $B\coloneqq \bAut^0(X, \De)$, is an abelian variety of dimension $h^1(X, \Oo_X)$.

\item {\rm (\cite[Theorem 4.5]{Xu20})}. The Albanese morphism $ X \to A$ is isomorphic to the homogeneous fibration induced by the
action of $B$ on $X$, thus $A = B/K$ for some finite subgroup $K \subset B$.

\item {\rm (\cite[Theorem 1.13]{Xu20})}. There exists an \'etale group homomorphism $\bAut^0(X,\De) \to A(X)$ (not canonical) such that
\[
X\simeq \bAut^0(X,\De) \times^K F,
\] where $K =\Ker(\bAut^0(X,\De) \to A(X))$ and $F$ is the fiber of the Albanese morphism $\alb_X$ over the identity element. Here $\bAut^0(X,\De) \times^K F$ denotes the quotient of $\bAut^0(X,\De) \times F$ by $K$ under the action $(xg,g^{-1}y)$, where $[g]\in K$ and $(x,y)\in \bAut^0(X,\De) \times F$ are closed points. 

\item {\rm (\cite[Theorem 3.1 (2)]{Xu20})}. The Albanese morphism $\alb_X: X \to A(X)$ is identified with the natural morphism
\[
\bAut^0(X,\De) \times^K F \to \bAut^0(X,\De)/K.
\]
\end{enumerate}
\end{theorem}

\begin{remark}
The original statement of \cite[Theorem 4.5]{Xu20} is for $\Qq$-divisors, but it still holds for $\Rr$-divisors. In fact, a klt pair $(X,\De)$ satisfying $K_X+\De \sim_\Rr 0$ implies the existence of a $\Qq$-divisor $\De'$ such that $(X,\De')$ is klt with $K_X+\De'\sim_\Qq 0$ and $\Supp \De= \Supp \De'$.
\end{remark}

The following is a version of a result of Blanchard about holomorphic transformation groups (see \cite{Bla56}). Blanchard's lemma will be used repetitively in the sequel and will be generalized in Lemma \ref{le: Blanchard}.

\begin{lemma}[{Blanchard's Lemma (\cite[Proposition 4.2.1]{BSU13})}]\label{le: Blanchard 0}
Let $f: X \to Y$ be a proper morphism of varieties over $k$ which is an algebraically closed field of characteristic $0$. Suppose that $f_*\Oo_X = \Oo_Y$. Let $G$ be a connected group scheme acting on $X$. Then there exists a unique $G$-action on $Y$ such that $f$ is $G$-equivariant. 

In particular, there exists a natural group homomorphism
\begin{equation}\label{eq: Blanchard}
\nu: \bAut^0(X) \to \bAut^0(Y), \quad [g] \mapsto [g_Y]
\end{equation} such that $g_Y\circ f = f\circ g$.
\end{lemma}

By our convention, if $f: X \to Y$ is a fibration, then $f_*\Oo_X=\Oo_Y$. To simplify the expression, for a map $\mu: A \to B$, if $A_0 \subset A$ or $\mu(A) \subset B_0 \subset B$, we still use $\mu$ to denote the natural restriction maps $A_0 \to B$, $A \to B_0$, etc.

\begin{lemma}\label{le: big divisor}
Let $(X, \De)$ be a projective log Calabi-Yau pair with klt singularities over an algebraically closed field of characteristic $0$, and $D$ be a nef and big Cartier divisor on $X$. Let $G \subset \bAut^0(X)$ be a connected sub-abelian variety, then
\begin{equation}\label{eq: theta}
\vartheta_D: G \to \bPic^0(X), \quad [g] \mapsto [g^*D -D]
\end{equation} is a homomorphism of algebraic groups whose kernel is finite.
\end{lemma}
\begin{proof}
By the universal property of the Picard functor, there is a morphism 
\[
\vartheta_D: G \to \bPic(X), \quad [g] \mapsto [g^*D -D].
\] As $G$ is connected and contains the identity element, $\vartheta_D (G) \subset \bPic^0(X)$. As $\vartheta_D([{\rm Id}]) = [0]$, $\vartheta_D$ is a group homomorphism (see \cite[Page 41, Corollary 1]{Mum70}).

By the base-point free theorem, $D$ is semi-ample. Let $h: X \to Y$ be the birational morphism induced by $D$. Let $D=h^*H$ with $H$ an ample Cartier divisor on $Y$. Moreover, $K_X+\De= h^*(K_Y+\De_Y)$ where $\De_Y= h_*\De$. Thus $(Y, \De_Y)$ is still a projective Calabi-Yau pair with klt singularities. In particular, $Y$ has rational singularities and thus $h^1(Y, \Oo_Y)=h^1(X, \Oo_X)$. By Blanchard's lemma, there is a morphism $\bAut^0(X,\De) \to \bAut^0(Y)$. It is straightforward to see that this is an injective map. Let $G_Y$ be the image of $G$ under this homomorphism. As the base field is of characteristic $0$, both $\bPic^0(X)$ and $\bPic^0(Y)$ are smooth varieties (\cite[Page 95, Theorem]{Mum70}). Hence $\dim \bPic^0(X) = h^1(X, \Oo_X)= h^1(Y, \Oo_Y) =\dim \bPic^0(Y)$ by \cite[Corollary 9.5]{Kle05}. Thus the natural injective morphism of algebraic groups
\[
\bPic^0(Y) \to \bPic^0(X), \quad [L] \mapsto [h^*L]
\] is an isomorphism. We have the following commutative diagram, where $\vartheta_H([g]) = [g^*H - H]$ with $H$ an ample divisor on $Y$:
\[
\xymatrix{
G \ar[d]_\simeq \ar[r]^{\vartheta_D} &\bPic^0(X) \\
G_Y \ar[r]^{\vartheta_H} & \bPic^0(Y).\ar[u]_\simeq}
\] 

To show that $\vartheta_D$ has a finite kernel, it suffices to show that $\vartheta_H$ has a finite kernel. This can be done similarly to \cite[Theorem 4.6]{Han87}. As $\vartheta_{kH} = (\vartheta_H)^{\otimes k}$, it suffices to show that $\vartheta_{kH}$ has a finite kernel. Assume that $kH$ is very ample, and $Y \to \Pp^n$ is the closed embedding induced by $|kH|$. If $\vartheta_{kH}([g])=[0]$, then $g^*(kH) \sim kH$. Hence $g$ induces an isomorphism $H^0(Y, \Oo_Y(g^*(kH))) \simeq H^0(Y, \Oo_Y(kH))$. Therefore, $g$ corresponds to an automorphism of $\Pp^n$ which leaves $Y$ invariant. In other words, we can identify $g$ with a closed point in $\bAut(\Pp^n, Y)$, where $\bAut(\Pp^n, Y)$ is the group scheme that represents the automorphisms of $\Pp^n$ leaving $Y$ invariant. Let $j$ be this natural group homomorphism 
\[
j: \vartheta_{kH}^{-1}([0]) \to \bAut(\Pp^n, Y).
\] If $j([g])=[h]$, then $h|_Y=g$. Thus $j$ is injective. Since $\bAut(\Pp^n, Y)$ is a subgroup of $\bAut(\Pp^n)$, $\vartheta_{kH}^{-1}([0])$ is a linear group scheme. On the other hand, $G_Y$ is an abelian variety as $G$ is an abelian variety. Hence $\vartheta_{kH}^{-1}([0])$ is a subgroup scheme of an abelian variety. Therefore, $\vartheta_{kH}^{-1}([0])$ must be a finite group. This completes the argument.
\end{proof}

Lemma \ref{le: big divisor} is well-known when $D$ is ample, and in this case, the Calabi-Yau condition is not necessary. We partially generalize Lemma \ref{le: big divisor} to a semi-ample divisor $D$.

\begin{theorem}\label{thm: general divisor}
Let $(X, \De)$ be a projective klt Calabi-Yau pair over an algebraically closed field of characteristic $0$. Let $f: X \to Y$ be a fibration to a normal projective variety. Let $H$ be a nef and big Cartier divisor on $Y$.

\begin{enumerate}
\item The natural morphism
\[
{\vartheta_H} \circ \nu: \bAut^0(X, \De) \to \bAut^0(Y) \to \bPic^0(Y), \quad [g] \mapsto [g_Y] \mapsto [g_Y^*H -H]
\] is surjective.
\item The image of 
\[
\vartheta_{f^*H}: \bAut^0(X, \De) \to \bPic^0(X), \quad [g] \mapsto [g^*(f^*H) - (f^*H)]
\] is $f^*\bPic^0(X) \coloneqq \{[f^*L] \in \bPic^0(X) \mid [L] \in \bPic^0(Y)\}$.
\end{enumerate}
\end{theorem}
\begin{proof}
We proceed with the argument in several steps.

Step 1. 

Let $\alb_X: X \to A(X)$ and $\alb_Y: Y \to A(Y)$ be the Albanese morphisms. By the universal property of Albanese morphisms, there exists $g: A(X) \to A(Y)$ such that the following diagram commutes
\begin{equation}\label{eq: commutativity of alb}
\xymatrix{
X \ar[d]_f \ar[r]^{\alb_X} &A(X)\ar[d]^g \\
Y \ar[r]^{\alb_Y} & A(Y).}
\end{equation} By the canonical bundle formula \cite{Amb05}, there exists a $\De_Y$ such that $(Y, \De_Y)$ is still a klt Calabi-Yau pair. By Theorem \ref{thm: Xu} (2), $\alb_X$ and $\alb_Y$ are both fibrations, and thus $g$ is also a fibration. By Blanchard's lemma again, there exists a group homomorphism
\[
\tau: \bAut^0(A(X)) \to \bAut^0(A(Y)).
\] 

We claim that $\tau$ is surjective. Let $\beta\in A(Y)$ be a closed point, and $[\beta] \in \bAut^0(A(Y))$ be the corresponding morphism $[\beta]: y \mapsto y+\beta$ for each $y\in A(Y)$. Take a closed point $y_0\in A(Y)$, and let $x_0, x_0'\in A(X)$ be closed points such that $g(x_0) = y_0, g(x_0')=y_0+\beta$, and $\alpha = x_0'-x_0$. Then we have $\tau([\alpha]) =[\beta]$. In fact, as $g$ is a homomorphism of abelian varieties, 
\[
g([\alpha](x)) = g(x+\alpha)=g(x)+g(x'_0)-g(x_0) = g(x)+\beta=[\beta](g(x)).
\] Hence $\tau([\alpha]) = [\beta]$ by the construction of $\tau$.

Step 2. 

By the universal property of the Albanese morphisms, an isomorphism of $X$ induces an isomorphism of $A(X)$. Hence there exist group homomorphisms
\[
\iota_X: \bAut^0(X) \to \bAut^0(A(X)), \quad \iota_Y: \bAut^0(Y) \to \bAut^0(A(Y)).
\] Moreover, these morphisms sit in the following diagram
\begin{equation}\label{eq: commutative 1}
\xymatrix{
\bAut^0(X) \ar[d]_{\nu} \ar[r]^{\iota_X} & \bAut^0(A(X))\ar[d]^\tau \\
\bAut^0(Y) \ar[r]^{\iota_Y} & \bAut^0(A(Y)).}
\end{equation} To check the commutativity of this diagram, it suffices to check the commutativity of the diagram below with $\phi \in \bAut^0(X)$,
\[
\xymatrix{A(X) \ar[d]_{\iota_X(\phi)} \ar[r]^{g} & A(Y)\ar[d]^{\iota_Y(\nu(\phi))} \\
A(X) \ar[r]^{g} & A(Y).}
\] This follows from \eqref{eq: commutativity of alb} and the surjectivity of $\alb_X, \alb_Y$.

By Theorem \ref{thm: Xu} (3), there is an \'etale group homomorphism $\bAut^0(X,\De) \to A(X)$ (not canonical) such that
\[
X\simeq \bAut^0(X,\De) \times^K F,
\] where $K =\Ker(\bAut^0(X,\De) \to A(X))$ and $F$ is the fiber of $\alb_X$ over the identity element. By Theorem \ref{thm: Xu} (4), the Albanese morphism $\alb_X: X \to A(X)$ is identified with the natural morphism
\[
\bAut^0(X,\De) \times^K F \to \bAut^0(X,\De)/K.
\] Therefore, the natural homomorphism
\[
\bAut^0(X,\De)\xhookrightarrow{\iota_X} \bAut^0(A(X)) \simeq \bAut^0( \bAut^0(X,\De)/K) \simeq  \bAut^0(X,\De)/K
\] is just the quotient map. In particular,
\begin{equation}\label{eq: iota_X surj}
\iota_X(\bAut^0(X,\De)) = \bAut^0(A(X)).
\end{equation}

By Step 1, $\tau: \bAut^0(A(X)) \to \bAut^0(A(Y))$ is surjective, hence, by \eqref{eq: commutative 1}, we have
\[
\iota_Y \circ \nu(\bAut^0(X,\De))= \tau\circ\iota_X(\bAut^0(X,\De)) = \tau(\bAut^0(A(X))) = \bAut^0(A(Y)). 
\] Let 
\[
G_Y\coloneqq \nu(\bAut^0(X,\De)) \subset \bAut^0(Y).
\] It is an abelian variety as $\bAut^0(X,\De)$ is an abelian variety by Theorem \ref{thm: Xu} (1). Then
\[
\dim G_Y \geq \dim \bAut^0(A(Y)) = \dim A(Y).
\] 

Step 3. 

By Lemma \ref{le: big divisor}, the group homomorphism
\[
G_Y \xrightarrow{\vartheta_H} \bPic^0(Y)
\] has a finite kernel. Thus
\[
\dim \vartheta_H(G_Y)=\dim G_Y \geq \dim A(Y) =\dim \bPic^0(Y).
\] Hence $\vartheta_H$ is surjective. This shows the first claim.

It is straightforward to check that the following diagram commutes, 
\[
\xymatrix{
\bAut^0(X, \De) \ar[d]_{\nu} \ar[r]^{\vartheta_{f^*H}} & \bPic^0(X) \\
\bAut^0(Y) \ar[r]^{\vartheta_H} & \bPic^0(Y)\ar[u]_{f^*}.}
\] Combining it with the first claim, we have the second claim.
\end{proof}

Recall that for a fibration $X \to S$, we set $\eta\in S$ as the generic point and $X_{\eta}$ as the generic fiber. By our convention, $\bAut(X_\eta) = \bAut(X_\eta/\eta)$ and $\bPic(X_\eta) =\bPic(X_\eta/\eta)$, and each $g\in \Aut(X/S)$ acts on $L \in \Pic(X/S)$ by $g \cdot L \coloneqq g_*L$, where the pushforward $g_*L$ is the same as $(g^{-1})^*L$.

\begin{lemma}\label{le: natural map}
Let $X \to S$ be a fibration. There exists a natural morphism
\[
\alpha:  \bAut(X_\eta) \times_\eta\bPic(X_\eta)  \to \bPic(X_\eta)
\] such that $\bAut(X_\eta)$ is a left action on $\bPic(X_\eta)$ as an $\eta$-group scheme. The same claim also holds for $\bAut(X_{\bar\eta})$ and $\bPic(X_{\bar\eta})$.
\end{lemma}
\begin{proof}
We only establish the first claim as the second claim can be shown similarly.

Let $T$ be a locally noetherian $\eta$-scheme, then by the definition of $\bAut(X_\eta)$ and $\bPic(X_\eta)$, we have 
\[
\Hom_\eta(T, \bAut(X_\eta)) \simeq \CAut_{X_\eta}(T)= \Aut(X_T/T)
\] and 
\[
\Hom_\eta(T, \bPic(X_\eta) \simeq \CPic_{X_\eta, \text{(\'et)}}(T).
\] For $L \in \CPic_{X_\eta, \text{(\'et)}}(T)$, by the definition of the sheafification of the Picard functor in the \'etale topology \cite[Definition 2.63]{Vis05}, there exist an \'etale covering $\{\sigma_i: U_i \to T\}$ and an $L_i\in \Pic(X_{U_i})/\Pic(U_i)$ such that $L_i =\sigma_i^*L$. Note that in the above expression, we identify $\Pic(X_{U_i})/\Pic(U_i)$ as a subset of $\CPic_{X_\eta, \text{(\'et)}}(T)$ by \cite[Theorem 2.64(iii)]{Vis05} as $\CPic_{X_\eta}$ is a separated functor (see \cite[Definition 2.37]{Vis05}) in the \'etale site. Let $g\in \Aut(X_T/T)$ and $g_i \in \Aut(X_{U_i}/U_i)$ be the pull-back of $g$. Note that for $F_i \in \Pic(U_i)$, if $\pi_i: X_{U_i} \to U_i$ is the natural morphism, then $(g_i)_*(\pi^*F_i) \in \pi_i^*(\Pic(U_i))$ as $\pi_i$ is a  fibration. Therefore, $g_i$ naturally acts on $L_i$ by pushing forward the line bundle in $\Pic(X_{U_i})$ that represents $L_i$. By abusing the notation, we write this action by $(g_i)_*L_i$. Then $(g_i)_*L_i|_{U_i \times_T U_j} = (g_j)_*L_j|_{U_i \times_T U_j}$ for each $i, j$. As $\CPic_{X_\eta, \text{(\'et)}}$ is a sheaf (see \cite[Definition 2.37]{Vis05}), there exists a unique $\ti L \in \CPic_{X_\eta, \text{(\'et)}}(T)$ such that $\ti L|_{U_i} =(g_i)_*L_i$. We write $g_*L$ for $\ti L$. 

By the definition of the fiber product,
\[
\Hom_\eta(T, \bAut(X_\eta)) \times \Hom_\eta(T, \bPic(X_\eta)) \simeq \Hom_\eta(T, \bAut(X_\eta) \times_\eta \bPic(X_\eta)).
\] Therefore, the above discussion gives natural maps
\begin{equation}\label{eq: natural map}
\begin{split}
& \Hom_\eta(T, \bAut(X_\eta) \times_\eta \bPic(X_\eta)) \\
\simeq &\Hom_\eta(T, \bAut(X_\eta)) \times \Hom_\eta(T, \bPic(X_\eta))  \xrightarrow{\tau} \Hom_\eta(T, \bPic(X_\eta))
\end{split}
\end{equation}
with $\tau(g, L) = g_*L$. Take $T=\bAut(X_\eta) \times_\eta \bPic(X_\eta)$, then the identity morphism in $\Hom_\eta(\bAut(X_\eta) \times_\eta \bPic(X_\eta), \bAut(X_\eta) \times_\eta \bPic(X_\eta))$ gives the natural morphism 
\[
\alpha: \bAut(X_\eta) \times_\eta \bPic(X_\eta) \to  \bPic(X_\eta/\eta).
\]

Next, we show that $\alpha$ is a group scheme action. Let $m: \bAut(X_\eta) \times_\eta\bAut(X_\eta) \to \bAut(X_\eta)$ be the multiplication morphism for the group scheme $\bAut(X_\eta)$, and ${\text{Id}}_A, {\text {Id}}_P$ be the identity morphisms on $\bAut(X_\eta), \bPic(X_\eta)$, respectively. For a morphism $\beta: A \to B$ over $\eta$, we use $\hat\beta: \Hom_\eta(T,A) \to \Hom_\eta(T,B)$ to denote the morphism obtained by applying $\Hom_\eta(T, -)$ to $\beta$.

For $g, h \in \Aut(X_T)$, we have $g_*(h_*L) = (g\circ h)_*L$. Therefore,
\[
\begin{split}
&\Hom_\eta(T,  \bAut(X_\eta) \times_\eta\bAut(X_\eta) \times_\eta\bPic(X_\eta)) \\
\xrightarrow{{(m \times \text{Id}_{P})}^{\wedge}} &\Hom_\eta(T,  \bAut(X_\eta)\times_\eta\bPic(X_\eta)) \xrightarrow{\hat\alpha}\Hom_\eta(T,  \bPic(X_\eta))
\end{split}
\] is the same as
\[
\begin{split}
&\Hom_\eta(T,  \bAut(X_\eta) \times_\eta\bAut(X_\eta) \times_\eta\bPic(X_\eta)) \\
\xrightarrow{{(\text{Id}_{A} \times \alpha)}^{\wedge}} &\Hom_\eta(T,  \bAut(X_\eta)\times_\eta \bPic(X_\eta)) \xrightarrow{\hat\alpha} \Hom_\eta(T,  \bPic(X_\eta)).
\end{split}
\] Let $T = \bAut(X_\eta) \times_\eta\bAut(X_\eta) \times_\eta\bPic(X_\eta)$, then the image of $\text{Id}_T: T \to T$ under $\hat\alpha\circ (m \times {\text Id}_{P})^\wedge$ is the same as $\alpha \circ (m \times \text{Id}_{P} )$. Similarly, the image of $\text{Id}_T$ under $\hat\alpha \circ {(\text{Id}_{A} \times \alpha)}^\wedge$ is the same as $\alpha \circ (\text{Id}_{A} \times \alpha)$. Thus $\alpha \circ (m\times\text{Id}_{P})=\alpha \circ (\text{Id}_{A} \times \alpha)$. Besides, it is straightforward to check that 
\[
\alpha(e,-): \{e\} \times_\eta  \bPic(X_\eta) \to  \bPic(X_\eta)
\] is an isomorphism, where $e$ is the identity element of the group scheme $\bAut(X_\eta)$. Hence $\alpha$ is a group scheme action.
\end{proof}

The following generalizes Blanchard's lemma (see Lemma \ref{le: Blanchard 0}) for fibrations over a possibly non-algebraically closed field of characteristic $0$. In the sequel, we will use $(-)_F$ to denote the base extension of an object (schemes, morphisms, etc.) to the field $F$.

\begin{lemma}\label{le: Blanchard}
Let $f: X \to Y/K$ be a proper morphism of schemes over a field $K$ of characteristic $0$. Let $\bar K$ be the algebraic closure of $K$, and $f_{\bar K}: X_{\bar K} \to Y_{\bar K}$ be the base extension of $f$ to $\bar K$. Suppose that $f_{\bar K *}\Oo_{X_{\bar K}}=\Oo_{Y_{\bar K}}$, then there exists a natural group homomorphism
\[
\nu: \bAut^0(X) \to \bAut^0(Y)
\] such that its base extension to $\bar K$ is exactly the group homomorphism \eqref{eq: Blanchard} in Blanchard's lemma.

Moreover, the morphism $\nu$ commutes with the base extension for an algebraic extension $F/K$ in the following sense: if 
\begin{equation}\label{eq: mu_F}
\tau: \bAut^0(X_F) \to \bAut^0(Y_F)
\end{equation} is the natural group homomorphism, then $\tau=\nu_F$ is the base extension of $\nu$ to $F$.
\end{lemma}
\begin{proof}
Let 
\[
\mu: \bAut^0(X_{\bar K}) \to \bAut^0(Y_{\bar K}), \quad [g] \mapsto [g_Y]
\] be the natural group homomorphism \eqref{eq: Blanchard} in Blanchard's lemma. Then $\Gal(\bar K/K)$ naturally acts on both $\bAut^0(X_{\bar K})$ and $\bAut^0(Y_{\bar K})$. We claim that $\mu$ is $\Gal(\bar K/K)$-invariant, that is, for any $[g]\in \bAut^0(X_{\bar K})$ and $\sigma \in \Gal(\bar K/K)$, we have $\mu(\sigma\cdot [g]) = \sigma\cdot \mu([g])$. Note that we have a commutative diagram
\[
\xymatrix{
X_{\bar K} \ar[d]^{f_{\bar K}} \ar[r]^g &X_{\bar K} \ar[r]^{\sigma}\ar[d]^{f_{\bar K}} & X_{\bar K}  \ar[d]^{f_{\bar K}}\\
Y_{\bar K} \ar[r]^{g_Y} & Y_{\bar K} \ar[r]^{\sigma} &Y_{\bar K}}\\
\] where the commutativity of the left square follows from Blanchard's lemma and the commutativity of the right square follows from that $f$ is defined over $K$. By the uniqueness in Blanchard's lemma, $\mu([\sigma\circ g])=[\sigma\circ g_Y]$. As $[\sigma\circ g]=\sigma\cdot [g]$ and $[\sigma\circ g_Y]=\sigma\cdot [g_Y]$, we have
\[
\mu(\sigma\cdot[g])=[\sigma\circ g_Y]=\sigma\cdot [g_Y]=\sigma\cdot \mu([g]).
\] This shows that $\mu$ is $\Gal(\bar K/K)$-invariant. As
\[
(\bAut^0(X))_{\bar K} = \bAut^0(X_{\bar K}) \text{~and~} (\bAut^0(Y))_{\bar K} = \bAut^0(Y_{\bar K})
\] by \cite[Lemma 9.5.1 (3)]{Kle05}. By the Galois descent for proper morphisms (See \cite[Corollary 11.2.9]{Spr09}), there exists a morphism
\[
\nu: \bAut^0(X) \to \bAut^0(Y)
\] such that its base extension to $\bar K$ is exactly $\mu$. Moreover, $\nu$ is a group homomorphism because $\mu$ is a group homomorphism.

To see that $\nu$ commutes with an algebraic base extension $F/K$, without loss of generality, we can assume that $F \subset \bar K$. The same argument as above shows that $\mu$ is $\Gal(\bar K/F)$-invariant and thus descents to a morphism $\bAut^0(X_F) \to \bAut^0(Y_F)$. By $\mu=\nu_{\bar K}= (\nu_F)_{\bar K}$, it must descent to $\nu_F$.
\end{proof}

Recall that $\Aut^0(X_\eta, \De_\eta) \coloneqq \bAut^0(X_\eta, \De_\eta)(\eta)$. The following lemma will be used in Section \ref{sec: fibration vs generic fiber}.

\begin{lemma}\label{le: same aut0}
Let $\phi: X \dto Y/S$ be a birational map between $\Qq$-factorial projective varieties that are isomorphic in codimension $1$. Let $\De$ be a divisor on $X$. Then there exists an isomorphism of groups
\[
\Aut^0(X_\eta, \De_\eta) \to \Aut^0(Y_\eta, \De_{Y,\eta}),\quad g \mapsto g_Y,
\] where $g_Y = \phi \circ g \circ \phi^{-1}$.
\end{lemma}
\begin{proof}
First, we show that there exists an isomorphism
\[
\bAut^0(X_{\bar\eta}, \De_{\bar\eta}) \to \bAut^0(Y_{\bar\eta}, \De_{Y,{\bar\eta}}),\quad [\bar g] \mapsto [\bar g_Y], 
\] where $\bar g_Y = \bar\phi \circ \bar g \circ \bar\phi^{-1}$ with $\bar\phi: X_{\bar\eta}\to X_{\bar\eta}$ induced from $\phi$.

Let $\bar H_Y$ be an ample divisor on $Y_{\bar\eta}$ and $\bar H$ be its strict transform on $X_{\bar\eta}$. Then because $\bar g$ and $\text{Id}$ lie in the same connected component of $\bAut(X_{\bar\eta}, \De_{\bar\eta})$, we have $\bar g \cdot \bar H \equiv \bar H$ on $X_{\bar\eta}$ and thus $\bar g_{Y} \cdot \bar H_Y \equiv \bar H_Y$ on $Y_{\bar\eta}$. In particular, $\bar g_{Y} \cdot \bar H_Y$ is ample. As $\bar \phi$ is isomorphic in codimension $1$, so is $\bar g_Y$. Hence the birational map $\bar g_{Y}$ is an isomorphism. Thus we have a natural map
\[
\varphi: \bAut^0(X_{\bar\eta}, \De_{\bar\eta}) \to \bAut(Y_{\bar\eta}, \De_{Y,{\bar\eta}}),\quad [\bar g] \mapsto [\bar g_Y].
\] As $\bAut^0(X_{\bar\eta}, \De_{\bar\eta})$ is a connected group variety containing the identity, we have $\varphi(\bAut^0(X_{\bar\eta}, \De_{\bar\eta}))\subset  \bAut^0(Y_{\bar\eta}, \De_{Y,{\bar\eta}})$. The same argument on $\phi^{-1}$ gives the desired isomorphism.

By \cite[Lemma 9.5.1 (3)]{Kle05}, we have
\begin{equation}\label{eq: base extension}
\bAut^0(X_{\bar\eta}, \De_{\bar\eta}) = \bAut^0(X_{\eta}, \De_{\eta})_{\bar\eta}\text{~and~} \bAut^0(Y_{\bar\eta}, \De_{Y,{\bar\eta}}) =\bAut^0(Y_{\eta}, \De_{Y,{\eta}})_{\bar\eta}.
\end{equation} Hence, for any $g\in \Aut^0(X_{\eta}, \De_{\eta})$, we have $[\bar g] \in \bAut^0(X_{\bar\eta}, \De_{\bar\eta})$, where $\bar g$ is the base extension of $g$ on $X_{\bar\eta}$. Thus, $[\bar g_Y] \in \bAut^0(Y_{\bar\eta}, \De_{Y,{\bar\eta}})$ by the previous argument. From this, we claim that the birational map $g_Y=\phi \circ g \circ \phi^{-1}$ is indeed an isomorphism. Let $H$ be an ample divisor on $X_\eta$ and $\bar H$ be the base extension of $H$ on $X_{\bar\eta}$. Thus
\[
(g_Y\cdot H)_{\bar\eta} = \bar g_Y \cdot \bar H.
\] We claim that $g_Y\cdot H$ is also ample. Replacing $H$ by a multiple, we can assume that $g_Y\cdot H$ is Cartier. If $\Ff$ is a coherent sheaf on $X_\eta$ with $\bar \Ff$ the base extension of $\Ff$ on $X_{\bar\eta}$, then as $\bar\eta \to \eta$ is a flat base extension, by \cite[III Proposition 9.3]{Har77}, we have
\[
H^i(X_\eta, \Ff \otimes \Oo_{X_\eta}(m(g_Y \cdot H)))_{\bar\eta} \simeq H^i(X_{\bar\eta}, \bar\Ff \otimes \Oo_{X_{\bar\eta}}(m(\bar g_Y \cdot \bar H))).
\] Therefore, $g_Y \cdot H$ is ample by the cohomological criterion of ampleness (see \cite[III Proposition 5.3]{Har77}). Then $g_Y$ is an isomorphism by the same reason as before. Thus $g_Y \in \Aut(Y_\eta, \De_\eta)$. Since $[\bar g_Y] \in \bAut^0(Y_{\bar\eta}, \De_{Y,{\bar\eta}})$, we have $g_Y \in \Aut^0(Y_\eta, \De_{Y,\eta})$ by \eqref{eq: base extension}. Applying the same argument to $Y \dto X/S$, it is straightforward to see that $\Aut^0(X_\eta, \De_\eta) \to \Aut^0(Y_\eta, \De_{Y,\eta})$ is an isomorphism of groups.
\end{proof}

The following theorem is the main result of this section.

\begin{theorem}\label{thm: num to linear}
Let $f: (X,\De) \to S$ be a klt Calabi-Yau fiber space, and $\pi: X \to Y/S$ be a fibration over $S$. Suppose that $H$ is a nef and big$/S$ Cartier divisor on $Y$, and $B = \pi^*H$. Then for any Cartier divisor $\xi$ on $Y_\eta$ such that $\xi \equiv 0$, there exist $h\in \Aut^0(X_\eta,\De_\eta)$ and $\ell\in \Zz_{>0}$ such that
\[
h \cdot B_\eta - B_\eta \sim \ell \pi_{\eta}^*\xi,
\] where $\pi_\eta$ is the base extension of $\pi$ to $\eta$.
\end{theorem}
\begin{proof}
As $\pi: X \to Y$ is a fibration between normal varieties, we have $\pi_*\Oo_X=\Oo_Y$. Since $\bar\eta \to S$ is a flat base extension, $r: Y_{\bar\eta} \to Y$ is also a flat morphism. Let $s: X_{\bar\eta} \to X$ and $\pi_{\bar\eta}: X_{\bar\eta} \to Y_{\bar\eta}$ be the natural morphisms. As cohomology commutes with a flat base extension, we have
\[
\pi_{\bar\eta*}(s^*\Oo_X)=r^*(\pi_*\Oo_X).
\] By $s^*\Oo_X = \Oo_{X_{\bar\eta}}, r^*\Oo_Y = \Oo_{Y_{\bar\eta}}$ and $\pi_*\Oo_X = \Oo_Y$, we have $\pi_{\bar\eta*}\Oo_{X_{\bar\eta}}= \Oo_{Y_{\bar\eta}}$. Hence the assumptions of Lemma \ref{le: Blanchard} are satisfied.

By Lemma \ref{le: natural map}, we have a natural group scheme action
\[
\bAut(Y_\eta)\times_\eta\bPic(Y_\eta) \to \bPic(Y_\eta).
\] As $[H_\eta] \in \bPic(Y_\eta)$ is an $\eta$-point, we have the corresponding morphism
\[
\alpha_{H_\eta}: \bAut(Y_\eta) \simeq \bAut(Y_\eta) \times_\eta \{[H_\eta]\} \to \bPic(Y_\eta).
\] Let 
\[
(+): \bPic(Y_\eta) \times_\eta \bPic(Y_\eta) \to  \bPic(Y_\eta)
\] be the addition of abelian schemes. As $[-H_\eta] \in \bPic(Y_\eta)$ is an $\eta$-point, we have the morphism which is ``adding by $[-H_\eta]$''
\[
 (+)_{[-H_\eta]}: \bPic(Y_\eta) \simeq\bPic(Y_\eta)\times_\eta \{[-H_\eta]\}  \xrightarrow{(+)} \bPic(Y_\eta).
 \] By Lemma \ref{le: Blanchard}, we have the natural group homomorphism
 \[
 \nu: \bAut^0(X_\eta,\De_\eta) \hookrightarrow \bAut^0(X_\eta)\to \bAut^0(Y_\eta).
 \] Composing the above morphisms, we get
 \[
  (+)_{[-H_\eta]} \circ \alpha_{H_\eta} \circ \nu: \bAut^0(X_\eta, \De_{\eta})\to\bAut^0(Y_\eta) \to \bPic(Y_\eta) \to\bPic(Y_\eta). 
 \]
 As $[{\rm Id}] \mapsto [0]$, and $\bAut^0(X_\eta, \De_{\eta})$ is connected, the image of $\bAut^0(X_\eta,\De_\eta)$ is contained in $\bPic^0(Y_\eta)$. We denote this morphism by
 \[
 \Uptheta_{H_\eta}: \bAut^0(X_\eta, \De_{\eta}) \to \bPic^0(Y_\eta). 
 \]
 
By \cite[Lemma 9.5.1 (3)]{Kle05}, we have 
\[
(\bAut^0(X_\eta, \De_{\eta}))_{\bar\eta} = \bAut^0(X_{\bar\eta}, \De_{\bar\eta}) \text{~and~} (\bPic^0(Y_\eta))_{\bar\eta} =\bPic^0(Y_{\bar\eta}).
\] Then it is straightforward to see that
\[
 (\Uptheta_{H_\eta})_{\bar\eta}: (\bAut^0(X_\eta, \De_{\eta}))_{\bar\eta}   \to(\bPic^0(Y_\eta))_{\bar\eta}, \quad [\bar g] \mapsto [\bar g_{Y_{\bar\eta}}^*H_{\bar\eta}-H_{\bar\eta}]
\] with $[g_{Y_{\bar\eta}}] = \nu_{\bar\eta}([\bar g])$ is exactly the morphism in Theorem \ref{thm: general divisor} (1). Thus $(\Uptheta_{H_\eta})_{\bar\eta}$ surjective by Theorem \ref{thm: general divisor} (1). Therefore, there exists a finite Galois extension $F/K(S)$ such that 
\[
\bAut^0(X_F, \De_F) \cap (\Uptheta_{H_F})^{-1}([\xi_F])
\] contains $(\spec F)$-points, where $X_F, \De_F, \Uptheta_{H_F}$ and $\xi_F$ correspond to $X_\eta, \De_\eta, \Uptheta_{H_\eta}$ and $\xi_\eta$ respectively after the base extension to $F$. Take any $(\spec F)$-point
\[
\alpha \in \left(\bAut^0(X_F, \De_F) \cap (\Uptheta_{H_F})^{-1}([\xi_F])\right)(F).
\] In what follows, we will construct a $g\in \Aut^0(X_\eta,\De_\eta)$ from $\alpha$.

The addition of abelian groups
\[
(+)_F: \bAut^0(X_F,\De_F) \times_{\spec F} \bAut^0(X_F,\De_F) \to \bAut^0(X_F,\De_F)
\] admits a $\Gal(F/K(S))$-action. Then $(+)_F$ is $\Gal(F/K(S))$-invariant because  $(+)_F$ is obtained by the base extension of $\bAut^0(X_\eta,\De_\eta) \times_\eta \bAut^0(X_\eta,\De_\eta) \to \bAut^0(X_\eta,\De_\eta)$ to $F$ which is defined over $K(S)$. Therefore,
\begin{equation}\label{eq: action on sum}
\sum_{i=1}^\ell \sigma \cdot \alpha_i = \sigma \cdot \left(\sum_{i=1}^\ell \alpha_i\right) \in \Aut^0(X_F,\De_F)
\end{equation} for any $\sigma \in \Gal(F/K(S))$ and $\alpha_i \in \Aut^0(X_F,\De_F)$. We emphasize that \eqref{eq: action on sum} is the summation in the abelian group $\Aut^0(X_F,\De_F)$.

Let
\begin{equation}\label{eq: g}
\ti g \coloneqq \sum_{\sigma\in \Gal(F/K(S))} \sigma\cdot\alpha \in \Aut^0(X_F,\De_F).
\end{equation} Then $[\ti g] \in \bAut^0(X_F,\De_F)$ is a $\Gal(F/K(S))$-invariant $(\spec F)$-point by \eqref{eq: action on sum}. Thus $[\ti g]$ descents to an $\eta$-point by the Galois descent (see \cite[Proposition 11.2.8]{Spr09}). To be precise, this means that there exists an $\eta$-point $[g] \in \bAut^0(X_\eta,\De_\eta)$ such that after the base extension to $F$, we have
\[
[\ti g]=[g_F]\in \bAut^0(X_F,\De_F)=\bAut^0(X_\eta,\De_\eta)_F.
\]

As usual, we still use $g$ to denote the automorphism corresponding to $[g] \in \bAut^0(X_\eta,\De_\eta)$. Because $H_\eta$ is defined over $\eta$, we have $\Uptheta_{H_F}([\sigma\cdot \alpha]) =\Uptheta_{H_F}([\alpha])=[\xi_F]$. Thus,  by \eqref{eq: g}, 
\begin{equation}\label{eq: image of ti g}
\Uptheta_{H_{F}}([\ti g]) = [\ell \xi_F],
\end{equation} where $\ell = |\Gal(F/K(S))|$. 

Let $g_{Y_\eta}\in \Aut^0(Y_\eta)$ such that $[g_{Y_\eta}] = \nu([g])$ and $\ti g_{Y_F} \in \Aut^0(Y_F)$ such that $[\ti g_{Y_F}] = \nu_F([\ti g])$. By Lemma \ref{le: Blanchard}, we have $\nu_F([\ti g]) = \nu([g])_F$, and thus
\[
\ti g_{Y_F} = (g_{Y_\eta})_F.
\] By \eqref{eq: image of ti g},
\[
\Uptheta_{H_{F}}([\ti g])=[{\ti g_{Y_F}}^*(H_{F})-H_{F}]=[\ell\xi_F] \in \bPic^0(Y_F).
\] That is,
\[
[{\ti g_{Y_F}}^*(H_{F})-H_{F}- \ell\xi_F]=0\in \bPic^0(Y_F).
\] Because
\[
[g_{Y_\eta}^*H_\eta- H_\eta-\ell \xi]_F= [{\ti g_{Y_F}}^*(H_{F})-H_{F}- \ell\xi_F]=0 \in \bPic^0(Y_F) = \bPic^0(Y_\eta)_F.
\] We see $[g_{Y_\eta}^*H_\eta- H_\eta-\ell \xi]=0\in \bPic^0(Y_\eta)$, and thus
\[
g_{Y_\eta}^*(H_{\eta})-H_{\eta}\sim \ell\xi.
\] As $\pi_\eta\circ g = g_{Y_\eta}\circ \pi_\eta $ (because it holds after base extension to $\bar \eta$ by Lemma \ref{le: Blanchard}) and $B_\eta= \pi_\eta^*(H_\eta)$, we have
\[
g^* B_\eta - B_\eta =\pi_\eta^*(g_{Y_\eta}^*(H_{\eta})-H_{\eta})\sim \ell\pi_\eta^*\xi.
\] Note that $g^*B_{\eta}= g^{-1} \cdot (B_{\eta})$ where $g^{-1} \in \Aut^0(X_\eta,\De_\eta)$, we have 
\[
g^{-1} \cdot B_\eta-B_\eta \sim \ell\pi_\eta^*\xi
\] as desired.
\end{proof}

\begin{remark}
A similar construction of $g$ appeared in \cite{Kaw97} for an ample divisor $B$. We thank Yong Hu for clarifying the argument in  \cite{Kaw97}. 
\end{remark}

\section{Relative and generic cone conjectures}\label{sec: fibration vs generic fiber}

In this section, we study the relationship between the cone conjecture of a Calabi-Yau fibration and that of its generic fiber.

Recall that a polyhedral cone is closed by definition and $\Gamma_B$ is the image of $\PsAut(X/S, \De)$ under the group homomorphism $\PsAut(X/S, \De) \to {\rm GL}(N^1(X/S)_\Rr)$. By Definition \ref{def: polyhedral type}, we set
\[
\Mov(X/S)_+ \coloneqq \Conv(\bMov(X/S) \cap N^1(X/S)_\Qq).
\] In the sequel, we need the following lemma:

\begin{lemma}\label{le: compare movable}
Let $\xi$ be a Cartier divisor on $X_\eta$. Then $\xi$ is movable on $X_\eta$ if and only if $\xi_{\bar\eta}$ is  movable on $X_{\bar\eta}$
\end{lemma}
\begin{proof}
As $\bar\eta \to \eta$ is a flat base extension, for any $m\in\Zz$, we have the natural isomorphism $H^0(X_{\eta}, m\xi)_{\bar\eta} \simeq H^0(X_{\bar\eta}, m\xi_{\bar\eta})$ by by \cite[III Proposition 9.3]{Har77}. Hence $|m\xi|_{\bar\eta} = |m\xi_{\bar\eta}|$ and 
\[
({\rm Bs}|m\xi|)_{\bar\eta}={\rm Bs}|m\xi_{\bar\eta}|,
\] where ${\rm Bs}|-|$ denotes the base locus of the linear system. Then the claim follows from the definition of movable divisors.
\end{proof}

Now we are ready to prove the first main result of this paper.

\begin{proof}[Proof of Theorem \ref{thm: main 1}]
Possibly enlarge $P_\eta$, we can further assume that $P_\eta$ is a rational polyhedral cone. By Lemma \ref{le: bir=pseudoauto}, $\PsAut(X/S,\De)=\PsAut(X_\eta, \De_\eta)$, and hence we use the same notation to denote the corresponding birational maps. The argument proceeds in several steps.

Step 1.

In this step, we lift $P_\eta$ to $\Mov(X/S)$ after some modifications on $P_\eta$.

For any $[\xi] \in P_\eta$ with $\xi \geq 0$, there exists a unique $D\geq 0$ on $X$ such that $D_\eta = \xi$ and $\Supp D$ does not have vertical components. We claim that if $\xi$ is a movable divisor, then $D$ is also movable on $X/S$. 

Possibly replacing $\xi$ by a multiple, we can assume that $\codim{\rm Bs}|\xi| \geq 2$. If $\xi \sim \xi'$ on $X_\eta$, then $\xi-\xi'= \di(\alpha)$ for some $\alpha \in K(X)$. If $D'$ is a divisor such that $D'_\eta = \xi'$, then $D-D'-\di(\alpha)$ is a vertical divisor. Note that if $E$ is a prime vertical divisor, then there exists a vertical divisor $E'\geq 0$ such that $-E \sim_\Qq E'/S$. Hence, $D \sim_\Qq D'+F/S$ with $F \geq 0$ a vertical divisor. As $\codim{\rm Bs}|\xi| \geq 2$, there exists $\xi^i \geq 0, i=1,\dots, k$ such that $\xi \sim \xi^i$ and
\[
\codim_{X_\eta}(\Supp \xi \cap \Supp \xi^1 \cap \cdots \cap \Supp \xi^k) \geq 2.
\] Then the above construction gives divisors $D^i \geq 0, i=1, \ldots, k$ such that $D^i_\eta=\xi^i$ and $D \sim_\Qq D^i/S$ (note that $\Supp D^i$ may have vertical components). Because $\Supp D$ does not have vertical components, 
\[
\codim_{X}(\Supp D \cap \Supp D^1 \cap \cdots \cap \Supp D^k) \geq 2.
\] Hence $D$ is a movable divisor on $X/S$, and this shows the claim.

By Lemma \ref{le: natural maps} (4), there exists a natural injective map
\begin{equation}\label{eq: generic to geometric}
\mu: N^1(X_\eta) \to N^1(X_{\bar\eta}), \quad [\xi] \mapsto [\bar\xi]
\end{equation} that sends $\Mov(X_\eta)$ to $\Mov(X_{\bar\eta})$ (but may not be surjective). Let $P_{\bar\eta}\coloneqq\mu(P_{\eta})$ and 
\[
G \coloneqq \{g_{\bar\eta} \mid g \in \PsAut(X_{\eta}, \De_{\eta})\} \subset \PsAut(X_{\bar\eta}, \De_{\bar\eta})
\] be the subgroup. By Lemma \ref{le: inclusion implies equal} (2), we can assume that there are effective divisors $\xi^i, i=1, \ldots, k$ on $X_{\eta}$ such that

\begin{enumerate}[label=(C\arabic*),ref=C\arabic*]
\item\label{C1} each $\xi^i_{\bar\eta}$ is movable,
\item \label{C2}$\Pi_{\bar\eta} =\Cone([\xi^i_{\bar\eta}] \mid i=1, \ldots, k) \subset P_{\bar\eta}$, and
\item $G \cdot \Pi_{\bar\eta} = \left(G\cdot P_{\bar\eta} \right)\cap \Mov(X_{\bar\eta})$.
\end{enumerate}

By Lemma \ref{le: compare movable}, each $\xi^i$ is also movable. Let $D^i$ be the unique divisor on $X$ such that $D^i_\eta = \xi^i_\eta$ and $\Supp D^i$ do not have vertical components. Then $D^i$ is movable on $X/S$ by Step 1. Let
\begin{equation}\label{eq: Pi}
\Pi \coloneqq   \Cone([D^i] \mid i=1, \ldots, k) \subset \Mov(X/S)
\end{equation} be a rational polyhedral cone.

Let $\Pi_\eta$ be the image of $\Pi$ under the natural map $N^1(X/S) \to N^1(X_{\eta})$ (see Lemma \ref{le: natural maps} (1)). We claim that
\begin{equation}\label{eq: equal on eta}
\PsAut(X_\eta, \De_\eta) \cdot \Pi_\eta = \Mov(X_{\eta}).
\end{equation}
In fact, by $\PsAut(X_\eta, \De_\eta) \cdot P_\eta \supset\Mov(X_\eta)$ and $G \cdot \Pi_{\bar\eta} = \left(G\cdot P_{\bar\eta} \right)\cap \Mov(X_{\bar\eta})$, we have
\[
G \cdot \Pi_{\bar\eta} \supset \mu (\Mov(X_\eta)).
\] As $\mu$ is injective by Lemma \ref{le: natural maps}, we have 
$\PsAut(X_\eta, \De_\eta) \cdot \Pi_\eta \supset \Mov(X_\eta)$. As $\Pi_\eta \subset \Mov(X_\eta)$ by construction, we have the desired equality.

Replacing $P$ by $\Pi$, we can assume that
\begin{equation}\label{eq: P}
P=\Cone([D^i] \mid i=1, \ldots, k) \subset \Mov(X/S).
\end{equation}

Step 2.

Let $V \subset \Eff(X/S)$ be the vector space generated by vertical divisors. Our goal is to enlarge $P$ to a rational polyhedral cone $\Pi$ in such a way that a rational polyhedral subcone $Q\subset \Pi +V$ fulfills the requirement of the theorem.

First, we show that it suffices to construct a rational polyhedral cone $\Pi \subset \Eff(X/S)$ so that for any $[B] \in \Mov(X/S)_\Qq$, if $[B_\eta]\in P_\eta$, then
$[B] \in (\Aut^0(X/S,\De) \cdot \Pi) +V$.

Indeed, for any $[D] \in \Mov(X/S)_\Qq$, by the definition of $P_\eta$, there exists $g \in \PsAut(X_\eta, \De_\eta)$ and $[B_\eta] \in P_\eta$ such that $g \cdot [B_\eta] = [D_\eta]$. Thus $[(g^{-1} \cdot D)_\eta] \in P_\eta$. By the above claim, $[g^{-1} \cdot D] \in (\Aut^0(X/S,\De)\cdot\Pi)+V$. As $\Aut^0(X/S,\De) \subset  \PsAut(X/S, \De)$ and $ \PsAut(X/S, \De) \cdot V = V$, we have $[D] \in \PsAut(X/S, \De) \cdot (\Pi+V)$. This shows $\PsAut(X/S, \De) \cdot (\Pi+V) \supset \Mov(X/S)_\Qq$. By Lemma \ref{le: inclusion implies equal} (1), there exists a rational polyhedral cone $Q \subset \Pi+V$ such that
\[
\PsAut(X/S, \De) \cdot (Q \cap N^1(X/S)_\Qq) = \Mov(X/S)_\Qq.
\]

In the subsequent steps, we will proceed to construct a $\Pi$ to satisfy the desired property in Step 2. 

Step 3.

To begin, we analyze the property in Step 2 for a single effective $\Qq$-Cartier divisor $B$ such that $[B]\in\Mov(X/S)_\Qq$ and $[B_\eta]\in P_\eta$. We will construct a rational polyhedral cone $\Pi^{h_i}_B$ such that if $D$ is an effective divisor with $D_\eta \equiv B_\eta$, then
\[
[D] \in (\Aut^0(X/S,\De)\cdot \Pi^{h_i}_B)+V.
\] It suffices to do this for an $mB$ with $m\in \Zz_{>0}$. Hence, replacing $B$ by a multiple, $B$ can be assumed to be a Cartier divisor.

By Theorem \ref{thm: HX13} and Lemma \ref{le: lift to iso in codim 1}, for some $1\gg \ep>0$, $(X/S,\De+\ep B)$ has a minimal model $\phi: X \dto Y/S$ such that $\phi$ is isomorphic in codimension $1$ and $B_Y$ is semi-ample$/S$. Let $\tau: Y \to Z/S$ be the contraction induced by $B_Y$. Possibly replacing $B$ by a multiple, we can further assume that $B_Y=\tau^*H$ for an ample divisor $H$ on $Z/S$. Because $D_{Y,\eta}\equiv B_{Y,\eta}$ on $Y_\eta$, take a sufficiently small open set $U\subset S$, we see that $D_Y|_U \equiv B_Y|_U$ is nef on $Y_U/U$ by Lemma \ref{le: natural maps}. Thus, $D_Y|_U$ is semi-ample$/U$ by Theorem \ref{thm: HX13}. In particular, $D_Y|U=\tau_U^*H'$ for an ample divisor $H'$ on $Z_U/U$. This implies 
\begin{equation}\label{eq: on generic fiber}
D_{Y,\eta}-B_{Y,\eta} = \tau_\eta^*(H_\eta-H'_\eta),
\end{equation} and $H_\eta-H'_\eta \equiv 0$ on $Z_\eta$. According to Theorem \ref{thm: num to linear}, there exist $g_Y\in \Aut^0(Y_\eta, \De_{Y,\eta})$ and $k\in\Zz_{>0}$ such that
\[
g_Y\cdot B_{Y,\eta} - B_{Y,\eta}  \sim k (D_{Y,\eta}-B_{Y,\eta}).
\] As $X_\eta \dto Y_\eta$ is isomorphic in codimension 1, we have \[
\Aut^0(X_\eta, \De_\eta) \simeq \Aut^0(Y_\eta, \De_{Y\eta}), \quad g \mapsto g_Y\coloneqq \phi\circ g\circ\phi^{-1}
\] by Lemma \ref{le: same aut0}. Hence
\[
g\cdot B_{\eta} - B_{\eta} \sim k (D_{\eta}-B_{\eta}),
\] which implies
\begin{equation}\label{eq: move by g}
g\cdot B - B \sim k (D-B) \mod (\text{vertical divisors}).
\end{equation}

We claim that under the natural map (see Lemma \ref{le: natural maps} (1))
\[
\pi: \Eff(X/S)/V \hookrightarrow N^1(X/S)/V \to N^1(X_\eta),
\] the preimage of $[B_\eta]\in N^1(X_\eta)$ is
\begin{equation}\label{eq: B+N}
\pi^{-1}([B_\eta]) = [B]+N,
\end{equation} where $N$ is a finite-dimensional vector space.

In fact, replacing $X, B_\eta$ by $Y, B_{Y,\eta}$, respectively, we can assume that $B_\eta$ is semi-ample over $\eta$, and $h: X_\eta \to Z_\eta$ is the contraction defined by $B_\eta$. We will show that 
\begin{equation}\label{eq: preimage of B}
\pi^{-1}([B_\eta]) = [B]+\Span_\Rr\{[h^*\xi] \mid \xi \equiv 0 \text{~on~} Z_\eta\}/V,
\end{equation} where, by abusing notation, $[h^*\xi]$ denotes the divisor class $[\Xi] \in N^1(X/S)/V$ such that $\Xi$ is a divisor on $X$ satisfying $h^*\xi \sim_\Qq \Xi_\eta$. This $[\Xi]$ depends uniquely on $\xi$. 

Suppose that $D \geq 0$ is a $\Qq$-Cartier divisor with $D_\eta\equiv B_\eta$ on $X_\eta$ (i.e., $\pi([D])=[B_\eta]$). By the same argument as before (see \eqref{eq: on generic fiber}), we have $[D] \in [B]+\Span_\Rr\{[h^*\xi] \mid \xi \equiv 0 \text{~on~} Z_\eta\}/V$. Conversely, any $\Qq$-Cartier divisor $D$ satisfying $[D]\in [B]+\Span_\Rr\{[h^*\xi] \mid \xi \equiv 0 \text{~on~} Z_\eta\}/V$ is $\Qq$-linearly equivalent to an effective divisor (as $B_\eta$ is the pullback of an ample divisor from $Z_\eta$). This shows the claim.

Let $D_i \geq 0, i=1, \cdots, t$ be Cartier divisors on $X$ such that $D_{i,\eta}\equiv B_\eta$ and
\[
[D_i]-[B], \quad i=1, \cdots, t
\] generate $N$ (they may not necessarily be a basis), where $N$ is defined in \eqref{eq: B+N}. By \eqref{eq: move by g}, let $h_i \in \Aut^0(X_\eta,\De_\eta)$ and $k_i \in \Zz_{> 0}$ such that
\[
h_i\cdot B-B\sim k_i(D_i-B) \mod (\text{vertical divisors}), \quad i=1, \cdots, t.
\] Define a rational polytope
\begin{equation}\label{eq: Q_B}
Q_B^{h_i} \coloneqq [B] + \sum_{i=1}^t [0,1][h_i \cdot B-B] \subset N^1(X/S).
\end{equation} Then $\Pi^{h_i}_B$ is defined to be the rational polyhedral cone
\[
\Pi^{h_i}_B \coloneqq \Cone(Q_B^{h_i} ).
\] Let $\ti Q_B^{h_i}$ and $\ti \Pi^{h_i}_B\subset N^1(X/S)/V$ be the images of $Q_B^{h_i}$ and $\Pi^{h_i}_B$ under the natural quotient map $N^1(X/S) \to N^1(X/S)/V$, respectively. 

We claim that 
\begin{equation}\label{eq: property of Pi_B}
\langle h_i \mid i=1, \cdots t \rangle \cdot \ti\Pi^{h_i}_B \supset \pi^{-1}([B_\eta]),
\end{equation} where $\langle h_i \mid i=1, \cdots t \rangle \subset \Aut^0(X/S, \De)$ is the subgroup generated by $h_i, i=1, \cdots t$. In particular, \eqref{eq: property of Pi_B} implies that $\Pi^{h_i}_B$ satisfies the desired property stated at the beginning of Step 3.

As in the proof of Theorem \ref{thm: num to linear}, there exists a natural group homomorphism
\[
\bAut^0(Y_\eta, \De_{Y,\eta}) \xrightarrow{\vartheta_{B_{Y,\eta}}}  \bPic^0(Y_\eta).
\] The corresponding map on $\eta$-points
\[
\Aut^0(Y_\eta, \De_{Y,\eta}) \to \Pic^0(Y_\eta),\quad h_Y \mapsto h_Y^*B_{Y,\eta}-B_{Y,\eta}
\] is also a group homomorphism of abelian groups. As $\Aut^0(X_\eta,\De_\eta) \simeq \Aut^0(Y_\eta, \De_{Y,\eta})$ by Lemma \ref{le: same aut0},
\[
\Aut^0(X_\eta, \De_{\eta}) \to \Pic^0(X_\eta), \quad h \mapsto h^*B_{\eta}-B_{\eta}
\] is also a group homomorphism of abelian groups. Because
\[
h\cdot B_{\eta}-B_{\eta} = (h^{-1})^*\cdot B_{\eta}-B_{\eta},
\] we see that
\begin{equation}\label{eq: action is group hom}
\Aut^0(X_\eta, \De_{\eta}) \to \Pic(X_\eta), \quad h \mapsto h\cdot B_{\eta}-B_{\eta} 
\end{equation} is still a group homomorphism of abelian groups.

If $\{[D_i]-[B] \mid i=1, \cdots, s\}$ is a basis of $N$, then $\langle h_i \mid i=1, \cdots s \rangle$ acts on $[B]+N$ by translation according to \eqref{eq: action is group hom}. Hence, the polytope $\ti Q_B^{h_i}$ tiles 
\[
[B]+N=\pi^{-1}([B_\eta])
\] under the action of $\langle h_i \mid i=1, \cdots s \rangle$. See Figure \ref{fig: 1} below. This established \eqref{eq: property of Pi_B}. 

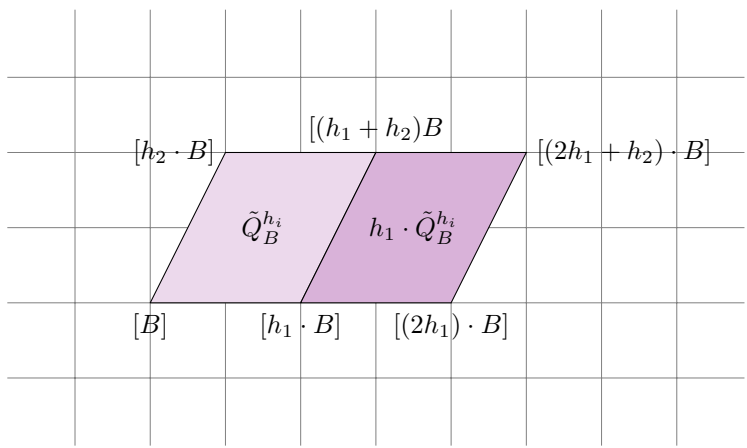
\begin{figure}[h]
  \centering
\begin{tikzpicture}
\draw[step=1cm,gray,very thin] (-1.9,-1.9) grid (7.9,3.9);
\filldraw[fill=violet!15!white, draw=black] (0,0) node[below] {\small $[B]$}-- (2,0) node[below] {\small $[h_1\cdot B]$}-- (3,2) node[pos=1, above] {\small $[(h_1+h_2)B$}--(1,2)node[left] {\small $[h_2\cdot B]$}--(0,0);
\filldraw[fill=violet!30!white, draw=black] (2,0) -- (4,0) node[below] {\small $[(2h_1)\cdot B]$}-- (5,2) node [right] {\small $[(2h_1+h_2)\cdot B]$}--(3,2)--(2,0);
\node at (1.5,1) {\small $\ti Q_B^{h_i}$};
\node at (3.5,1) {\small $h_1\cdot \ti Q_B^{h_i}$};
\end{tikzpicture}
\caption{The polytope $\ti Q_B^{h_i}$ tiles the affine space $\pi^{-1}([B_\eta]) = [B]+N$.}
\label{fig: 1}
\end{figure}

In summary, the above argument shows that as long as 
\[
[D_i]-[B], \quad i=1, \cdots, t
\] generate $N$, then
\begin{equation}\label{eq: generate W_B}
\Aut^0(X/S,\De) \cdot \ti\Pi^{h_i}_B \supset \pi^{-1}([B_\eta]).
\end{equation}

Step 4.

In this step, we globalize the construction in Step 3 for all divisors on $P_\eta$ uniformly. 

For each $g\in \Aut^0(X_\eta, \De_\eta)$ and a Cartier divisor $D$ on $X$, set
\begin{equation}\label{eq: star action}
g\star D \coloneqq [g\cdot D -D] \in N^1(X/S)_\Zz/(V\cap N^1(X/S)_\Zz).
\end{equation}

This action extends to a linear map
\[
G_g: \bigoplus_{i=1}^k \Qq \cdot D^i \to N^1(X/S)_\Qq/(V\cap N^1(X/S)_\Qq). 
\] as follows: for $D \coloneqq\sum_{i=1}^k a_i D^i$ with $a_i\in \Qq$, set
\[
G_g(\sum_{i=1}^k a_i D^i) = \sum_{i=1}^k a_i (g\star D^i) \in N^1(X/S)_\Qq/(V\cap N^1(X/S)_\Qq).
\]

As $\Aut^0(X/S, \De)=\Aut^0(X_\eta, \De_\eta)$ is an abelian group, and 
\begin{equation}\label{eq: star gp hom}
\Aut^0(X_\eta, \De_\eta) \to \Pic(X_\eta), \quad h \mapsto h\star D_\eta 
\end{equation} is a group homomorphism for a Cartier divisor $D$ (see \eqref{eq: action is group hom}). We have
\[
G_{g+h} = G_g + G_h.
\] This follows from $(G_{g+h}(D))_\eta = G_{g+h}(D_\eta)$ and
\[
G_{g+h}(D_\eta) = G_{g}(D_\eta)+G_{h}(D_\eta) \in \Pic(X_\eta).
\] Hence $G_{g+h} = G_g + G_h$.

Let $\Aut^0(X/S, \De)_\Qq \coloneqq \Aut^0(X/S, \De)\otimes_\Zz \Qq$ be a $\Qq$-vector space. For $\tau = \sum_j r_j g_j \in \Aut^0(X/S, \De)_\Qq$, set
\[
G_\tau(D)= \sum_j r_j (g_j \star D)\in N^1(X/S)/V.
\] We need to show that this is well-defined. In fact, suppose that $\tau = \sum_t r'_t g'_t \in \Aut^0(X/S, \De)_\Qq$. Choose $m \in \Zz_{>0}$ such that $mr_j, mr_t' \in \Zz$ for all $j, t$. Then $m \tau \in \Aut^0(X/S, \De)$. Hence, by the group homomorphism \eqref{eq: star gp hom},
\[
(mr_jg_j) \star D^i = mr_j(g_j \star D^i),\quad (mr_t'g_t') \star D^i = mr_t'(g_t' \star D^i), \text{~and}
\]
\[
G_{m\tau}(D) = \sum_{i=1}^k a_i((m \tau) \star D^i) = \sum_{i=1}^k a_i \left(\sum_j mr_j(g_j \star D^i)\right) = \sum_{i=1}^k a_i \left(\sum_t mr'_t(g'_t \star D^i)\right).
\] Therefore, $G_\tau$ is independent of the choice of the expression of $\tau$. Thus, we obtain a map
\[
\Aut^0(X/S, \De)_\Qq \times \bigoplus_{i=1}^k \Qq \cdot D^i \to N^1(X/S)_\Qq/V, \quad (\tau, D) \mapsto G_\tau(D),
\] and the natural corresponding map
\[
j: \Aut^0(X/S, \De)_\Qq \to \Hom_\Qq(\bigoplus_{i=1}^k \Qq \cdot D^i, N^1(X/S)_\Qq/V)
\] is linear. Choose $g_1, \cdots, g_\ell \in \Aut^0(X/S,\De)$ such that $\{G_{g_1}, \cdots, G_{g_\ell}\}$ is a basis of $j(\Aut^0(X/S, \De)_\Qq)$. By \eqref{eq: Q_B}, for each $D_i$, there is a rational polytope $Q_{D_i}^{g_j}$ associated to $g_1, \cdots, g_\ell$. Let 
\begin{equation}\label{eq: lambda}
\Lambda \coloneqq \sum_{1 \leq i \leq k}[0,1] Q_{D_i}^{g_j} \subset N^1(X/S)
\end{equation} be a rational polytope. We claim that the rational polyhedral cone
\begin{equation}\label{eq: Pi}
\Pi \coloneqq \Cone(\Lambda)
\end{equation} satisfies the requirement. Recall, this means that for any $[D] \in \Mov(X/S)_\Qq$, if $[D_\eta]\in P_\eta$, then
$[D] \in \left(\Aut^0(X/S,\De) \cdot \Pi\right) +V$. 

Let
\begin{equation}\label{eq: B}
B = \sum_i r_i D^i, \quad r_i \in [0,1]\cap \Qq.
\end{equation}
By Step 3, for each $B$, there exists $\{h_i \in \Aut^0(X/S,\De) \mid 1 \leq i \leq t\}$ such that 
\[
\{[h_i\cdot B-B]\mid i=1, \cdots, t\}
\] generates $N$. By the choice of $g_i$, we have
\[
\Span_\Qq\{G_{g_i}\mid 1 \leq i \leq \ell\} \supset\Span_\Qq\{G_{h_i}\mid 1 \leq i \leq t\}.
\] Thus
\[
\Span_\Qq\{G_{g_i}(B) \mid 1 \leq i \leq \ell\} \supset \Span_\Qq\{G_{h_i}(B)\mid 1 \leq i \leq t\}.
\] That is,
\[
\{[g_j\cdot B-B]\mid j=1, \cdots, \ell\}
\] generates $N$. Therefore, as shown by Step 3, we have
\[
\Aut^0(X/S,\De) \cdot \ti\Pi^{g_i}_B \supset \pi^{-1}([B_\eta]).
\]

Hence, to show the claim for $\Pi$, it suffices to show $\Pi \supset \Pi^{g_i}_B$ for any  $B = \sum_i r_i D^i, r_i \in \Qq_{\geq 0}$. In fact, if $[D]\in \Mov(X/S)_\Qq$ satisfies $[D_\eta] \in P_\eta$, then there exists a $B= \sum_i r_i D^i, r_i \in \Qq_{\geq 0}$ such that $D_\eta \equiv B_\eta$, that is, $[D] \in \pi^{-1}([B_\eta])$.

By \eqref{eq: B},
\[
g_j \cdot B  = \sum_i r_i (g_j \cdot D^i),
\] and thus for $0\leq \mu_j \leq 1, 1 \leq j \leq \ell$, we have
\[
B+\sum_j \mu_j(g_j \cdot B-B) =\sum_i r_i\left(D^i+\sum_j\mu_j(g_j\cdot D^i-D^i)\right).
\] By the construction of $Q^{g_i}_B$ (see \eqref{eq: Q_B}) and $\Lambda$ (see \eqref{eq: lambda}),
we have $\Pi^{g_i}_B \subset \Pi$. 

Finally, by Lemma \ref{le: inclusion implies equal} (1), there exists a rational polyhedral cone $Q \subset \Lambda\cap\Mov(X/S)$ such that $\PsAut(X,\De) \cdot (Q \cap N^1(X/S)_\Qq)= \Mov(X/S)_\Qq$. This completes the proof.
\end{proof}

\begin{remark}
Let $V, W$ be $\Qq$-vector spaces, and $O \subset \Hom_\Qq(V, W)$ be a subspace. Fix an affine space $L \subset V$. Suppose that for each $v\in L$, we have $\Span_\Qq\{Gv \mid G \in O\} = W$. This does not imply that the same property holds for $L_\Rr$. That is, there may exist $v \in L_\Rr -L$ such that $\Span_\Rr\{Gv \mid G \in O\} \neq W_\Rr$. We have the following example:

\begin{example}
Let $V=W = \Qq^2$, $O= \Span_\Qq\{
G_1=\begin{pmatrix}
1 & 2 \\
3 & 4
\end{pmatrix}, G_2=\begin{pmatrix}
1 & 0 \\
0 & 1
\end{pmatrix}\}$, and $L=\{x=-1\}$. Then one can check 
\[ 
\dim_\Rr\Span_\Rr\{Gv \mid G \in O\} = \left\{ 
\begin{array}{ll}
         1 & v=(-1, \frac{\pm\sqrt{33}-3}{4}),\\
        2 & v \in L \text{~and~} v \neq (-1, \frac{\pm\sqrt{33}-3}{4}).
\end{array} \right. 
\] 
\end{example}

Thus, we cannot obtain $\Aut^0(X/S,\De) \cdot Q = \Mov(X/S)$ from $\Aut^0(X/S,\De) \cdot (Q\cap N^1(X/S)_\Qq) = \Mov(X/S)_\Qq$.
\end{remark}

\section{Finiteness of contractions, minimal models, and the existence of weak fundamental domains}\label{sec: applications}

\subsection{Finiteness of contractions and minimal models}

The contractions of varieties are governed by the nef cones rather than movable cones. Nevertheless, the finiteness of (targets of) the contractions can be derived from the movable cone conjecture as stated in Theorem \ref{thm: finite contractions}. 

\begin{proof}[Proof of Theorem \ref{thm: finite contractions}]
Note that
\[
C\coloneqq\Cone\{g^*H \mid g: X \to Z/S, H \text{~is an ample}/S\text{~Cartier divisor on~} Z \} \subset \Mov(X/S).
\] By Lemma \ref{le: inclusion implies equal} (1), there exists a rational polyhedral cone $Q \subset \Mov(X/S)$ such that
\[
\PsAut(X/S,\De) \cdot (Q\cap N^1(X/S)_\Qq) =\Mov(X/S)_\Qq  \supset C\cap N^1(X/S)_\Qq.
\] Let $Q = \sqcup_{i=1}^m Q_i^\circ$ be the decomposition as in Theorem \ref{thm: Shokurov-Choi} such that for effective divisors $B, D$ with $[B], [D] \in Q_i^\circ$ and $1\gg\ep {> 0}$, the klt pairs $(X, \De+\ep B)$ and $(X, \De+\ep D)$ share same minimal models. Let $\sigma_i: X \dto Y_i$ be a minimal model that corresponds to $Q^\circ_i$. By Lemma \ref{le: lift to iso in codim 1}, we can further assume that $\sigma_i$ is isomorphic in codimension 1. Let $Q_i = \overline{Q_i^\circ}$ be the closure of $Q_i^\circ$, then the strict transform of the rational polyhedral cone $Q_i$, denoted by $Q^{Y_i}_i$, is contained in $\bAmp(Y_i/S)$. Hence, there are finitely many contractions $h^{i_k}_i: Y_i \to Z^{i_k}_i/S$ corresponding to the faces of $Q^{Y_i}_i$. We claim that if $g: X \to Z/S$ is a contraction, then $Z$ must be isomorphic to one of $Z^{i_k}_i$.

Let $H$ be an ample divisor on $Z$. Then $g^*H \in C$. Hence there exists $\tau \in \PsAut(X/S,\De)$ such that $\tau_*(g^*H) \in Q$. Suppose that $\tau_*(g^*H) \in Q^\circ_i$, then $\sigma_{i*}(\tau_*(g^*H))$ is nef$/S$ on $Y_i$. By Theorem \ref{thm: HX13}, $\sigma_{i*}(\tau_*(g^*H))$ is semi-ample$/S$. Thus there exists some $i_k$ such that 
\[
Z^{i_k}_i = \proj_S R(Y_i, n\sigma_{i*}(\tau_*(g^*H)))
\] for a fixed $n \in \Zz_{>0}$ such that $n\sigma_{i*}(\tau_*(g^*H))$ is Cartier. Here
\[
R(Y_i, n\sigma_{i*}(\tau_*(g^*H))) \coloneqq \bigoplus_{m\in \Zz_{\geq 0}}H^0(Y_i, \Oo_{Y_i}(mn\sigma_{i*}(\tau_*(g^*H)))).
\] Similarly, $Z = \proj_S R(X, ng^*H)$. As 
\[
\sigma_i \circ \tau: X \dto X \dto Y_i
\] is a composition of birational maps that are isomorphic in codimension 1, we have
\[
R(X, ng^*H) \simeq R(X, n\tau_*g^*H) \simeq R(Y_i, n\sigma_{i*}(\tau_*(g^*H))).
\] Therefore, $Z \simeq Z^{i_k}_i/S$.
\end{proof}

\begin{remark}\label{rmk: not the same as nef cone conj}
Theorem \ref{thm: finite contractions} is weaker than the consequence of the cone conjecture for nef cones which predicts that the contractions $\{X \to Z/S\}$ are finite up to isomorphisms on $X/S$. In the proof of Theorem \ref{thm: finite contractions}, we actually showed that $\{X \to Z/S\}$ are finite up to pseudo-automorphisms on $X/S$.
\end{remark}

\begin{proof}[Proof of Corollary \ref{cor: finite contractions generic case}]
By Theorem \ref{thm: main 1}, there exists a rational polyhedral cone $Q \subset \Mov(X/S)$ such that $\PsAut(X/S,\De) \cdot Q \supset\Mov(X/S)_\Qq$. Then the claim follows from Theorem \ref{thm: finite contractions}.
\end{proof}

Theorem \ref{thm: main 1} does not imply the existence of a weak fundamental domain directly (this will be addressed in Section \ref{subsec: fundamental domain} using deep results on the geometry of convex cones). But it at least implies the finiteness of models assuming that good minimal models exist.

\begin{proof}[Proof of Theorem \ref{thm: main 2}]
By Theorem \ref{thm: main 1}, there exists a rational polyhedral cone $Q\subset \Mov(X/S)$ such that $\PsAut(X,\De)\cdot Q \supset \Mov(X/S)_\Qq$. Let $Q = \sqcup_{i=1}^m Q_i^\circ$ be the decomposition as in Theorem \ref{thm: Shokurov-Choi} such that for effective divisors $B, D$ with $[B], [D] \in Q_i^\circ$ and $1\gg\ep {> 0}$, the klt pairs $(X, \De+\ep B)$ and $(X, \De+\ep D)$ share same minimal models. Let $\sigma_i: X \dto Y^i$ be a minimal model that corresponds to $Q^\circ_i$. By Lemma \ref{le: lift to iso in codim 1}, there exist a birational map $g_i: X \dto W/S$ that is isomorphic in codimension $1$ and a morphism $\nu: W \to Y^i$ such that $\nu\circ g_i = \sigma_i$.

First, we show that for any $\Qq$-factorial variety $W$ such that $X \dto W/S$ is isomorphic in codimension $1$, then there exists $i, 1\leq i \leq m$ such that $W \simeq W^i$. In fact, let $H$ be an ample$/S$ Cartier divisor on $W$. Let $H_X$ be the strict transform of $H$ on $X$. Then there exists $\tau \in \PsAut(X,\De)$ such that $\tau_*H_X \in Q$ and thus $\tau_*H_X \in Q^\circ_i$ for some $i, 1\leq i \leq m$. Then by the definition of $Q^\circ_i$ and $W^i$, we see that $g_{i*}(\tau_*H_X)$ is nef$/S$. Note that the nef$/S$ divisor $g_{i*}(\tau_*H_X)$ is also the strict transform of the ample$/S$ divisor $H$ through the birational maps 
\[
W \dto X \xdashrightarrow{\tau} X  \xdashrightarrow{g_i} W_i/S
\] that are all isomorphic in codimension $1$. As $W, W_i$ are both $\Qq$-factorial varieties, we have $W \simeq W_i/S$.

Next, as $g_i$ is isomorphic in codimension $1$, if $P^i_\eta$ is the strict transform of $P_\eta$ on $W^i_\eta$, then we still have
\[
\PsAut(W^i_\eta, \De_{W^i,\eta})\cdot P^i_\eta \supset  \Mov(W^i_\eta).
\] Therefore, by Corollary \ref{cor: finite contractions generic case},
\[
\{Z \mid W^i \to Z/S \text{~is a contraction,~} 1 \leq i \leq m\}
\] is a finite set. By the first part, $Y^i$ must belong to this finite set.
\end{proof}

\begin{remark}\label{rmk: terminalization}
In practice, for a klt Calabi-Yau fibration $f: (X,\De) \to S$, one can first take a terminalization $(\ti X/S, \ti \De)$ of $(X/S,\De)$. Then because a minimal model of $(X/S,\De)$ is also a minimal model of $(\ti X/S, \ti \De)$, one can apply Theorem \ref{thm: main 2} to $(\ti X/S, \ti \De)$ to obtain the finiteness of minimal models of $(X/S,\De)$.
\end{remark}

This following result generalizes both \cite{Kaw97} and \cite{FHS21} where the finiteness of minimal models are established for threefolds with $\dim(X/S) \leq 2$, and for elliptic fibrations, respectively. 

\begin{corollary}\label{cor: surface fibration}
Let $f: X \to S$ be a canonical Calabi-Yau fibration. Suppose that $\dim(X/S) \leq 2$, then $X/S$ has finitely many minimal models.
\end{corollary}
\begin{proof}
Let $\nu: \ti X \to X/S$ be a terminalization of $X$. Then $X$ is a $\Qq$-factorial terminal variety and $K_{\ti X}=\nu^*K_X$. As a minimal model of $X/S$ is also a minimal model of $\ti X/S$, replacing $X$ by $\ti X$, we can assume that $f: X \to S$ is a $\Qq$-factorial terminal Calabi-Yau fibration (see Remark \ref{rmk: terminalization}). When $\dim(X_\eta)\leq 2$, there exists a rational polyhedral cone $P_\eta \subset \Eff(X_\eta)$ such that 
\[
 \PsAut(X_\eta)\cdot P_\eta =  \bMov(X_\eta) \cap \Eff(X_\eta) \supset \Mov(X_\eta).
\] In fact, this automatically holds when $\dim(X_\eta)\leq 1$, and follows from the Morrison-Kawamata cone conjecture for surfaces when $\dim(X_\eta)\leq 2$ (see \cite[Remark 2.2]{Kaw97}). Besides, it is known that good minimal models exist for effective klt pairs in dimension $\leq 3$. Then the claim follows from Theorem \ref{thm: main 2}.
\end{proof}

Corollary \ref{cor: surface fibration} is also expected to hold for klt pairs. For this purpose, it suffices to show the Morrison-Kawamata cone conjecture for terminal log Calabi-Yau surfaces over a non-algebraically closed field of characteristic $0$. This case will be addressed elsewhere.

Using Corollary \ref{cor: surface fibration}, we can show that if $\dim X - \kappa(X) \leq 2$, then $X$ admits only finitely many minimal models. Here $\kappa(X)$ is the Kodaira dimension of $X$. Note that such an $X$ may not necessarily be a Calabi-Yau variety. It is straightforward to generalize this result to the relative setting, however, we just state the absolute case for the sake of simplicity.

\begin{proof}[Proof of Corollary \ref{cor: coKodaira leq 2}]
By \cite{BCHM10}, $X$ admits a canonical model $f: X \dto S$ where $S = \proj \oplus_{m\in\Zz_{\geq 0}} H^0(X, mnK_X)$ for some $n\in \Zz_{> 0}$ such that $nK_X$ is Cartier. Moreover, $\dim S = \kappa(X)$. Let $p_1: W \to X$ be a birational morphism from a smooth variety $W$ such that $q_1 \coloneqq f \circ p_1: W \to S$ is a morphism. Moreover, $q_1$ is a fibration by the definition of the canonical model. Then $K_W=p_1^*K_X+E$ with $E \geq 0$ a $p_1$-exceptional divisor. If $X \dto Y$ is a minimal model of $X$, then it is a standard fact that $Y$ is also a minimal model of $W$ under the natural map $W \to X \dto Y$.

As $\dim(W/S)=\dim X -\kappa(X) \leq 2$, by \cite[Theorem 0.2]{Lai11}, $W$ has a good minimal model $\theta: W \dto Y/\spec \Cc$. Thus, the semi-ample divisor $K_Y$ induces the natural morphism $Y \to S$ as $S$ is also the canonical model of $Y$. Moreover, $Y$ has canonical singularities and $K_Y \sim_\Qq 0/S$. Let $\mu: W \dto Y'$ be another minimal model of $W$. Suppose that $h: T \to W$, $p: T \to Y$ and $q: T \to Y'$ are birational morphisms such that
\[
\theta \circ h=p, \quad \mu\circ h = q.
\] Then
\[
h^*K_W=p^*K_Y+E, \quad h^*K_W=q^*K_{Y'}+F,
\] where $E \geq 0$ and $F \geq 0$ are $p$-exceptional and $q$-exceptional divisors, respectively. As $K_Y$ and $K_{Y'}$ are nef, $E$ is the negative part in the Nakayama-Zariski decomposition of $p^*K_Y+E$. Similarly, $F$ is the negative part in the Nakayama-Zariski decomposition of $q^*K_{Y'}+F$. Therefore, $E=F$ and 
\begin{equation}\label{eq: same pullback}
p^*K_Y=q^*K_{Y'}.
\end{equation} In particular, $K_{Y'}$ is also semi-ample, and it incudes the natural morphism $Y' \to S$.

Let $\tau: \ti Y \to Y$ be a $\Qq$-factorial terminalization of $Y$, we claim that $\ti\mu: \ti Y \dto Y'$ does not extract divisors. Thus $Y'$ is a minimal model of $\ti Y$ and also a minimal model of $\ti Y$ over $S$. Let $\ti p: \ti T \to \ti Y$ and $q': \ti T \to Y'$ be birational morphisms from a smooth variety $\ti T$ such that $\ti\mu\circ\ti p =q'$. By \eqref{eq: same pullback},
\begin{equation}\label{eq: same pullback 2}
\ti p^*K_{\ti Y} = \ti p^*(\tau^*K_{Y})=q'^*K_{Y'}.
\end{equation} Write $K_{\ti T}=\ti p^*K_{\ti Y} +\ti E$ and $K_{\ti T}=q'^*K_{Y'} +E'$. Then $\ti E$ and $E'$ are $\ti p$-exceptional and $q'$-exceptional divisors respectively. Moreover, as $\ti Y$ has $\Qq$-factorial terminal singularities, $\Supp E = \Exc(\ti p)$. By \eqref{eq: same pullback 2}, $\ti E = E'$. Thus any $\ti p$-exceptional divisor is also $q'$-exceptional. This shows that $\ti \mu: \ti Y \to Y'$ does not extract divisors.

It follows from Corollary \ref{cor: surface fibration} that the terminal Calabi-Yau fibration $\ti Y \to S$ has only finitely many minimal models. This completes the proof.
\end{proof}

\subsection{Existence of weak rational polyhedral fundamental domains}\label{subsec: fundamental domain}

Finally, we address the existence of weak rational polyhedral fundamental domains for $\Mov(X/S)_+$. Since $\Mov(X/S)_+ \not\subset \Eff(X/S)$ in general, the existence of weak rational polyhedral fundamental domains does not imply the finiteness of minimal models by the Shokurov polytope argument. However, when $R^1f_*\Oo_X =0$, we do have $\Mov(X/S)_+ = \Mov(X/S)\subset \Eff(X/S)$ by \cite[Proposition 5.3]{LZ22}.

\begin{proof}[Proof of Theorem \ref{thm: weak fundamental domains}]
Let $W$ be the maximal vector space in $\bMov(X/S)$. By Proposition \ref{prop: defined over Q}, $W$ is defined over $\Qq$. For any set $T\subset N^1(X/S)$, let $\widetilde T$ be the image of $S$ under the quotient map $N^1(X/S) \to N^1(X/S)/W$.

By Theorem \ref{thm: main 1}, there exists a rational polyhedral cone $Q\subset \Mov(X/S)$ such that
\[
\PsAut(X,\De) \cdot (Q \cap N^1(X/S)_\Qq)= \Mov(X/S)_\Qq.
\] Therefore,
\[
\PsAut(X,\De) \cdot \widetilde Q \supset \widetilde{\Mov(X/S)}_\Qq.
\] Note that $\widetilde{\Mov(X/S)}_\Rr$ is non-degenerate and $\widetilde Q \subset  \widetilde{\Mov(X/S)}_+$. By Proposition \ref{prop: prop-def} (3), we see that $(\widetilde{\Mov(X/S)}_+, \ti \Gamma_B)$ is of polyhedral type, where $\ti \Gamma_B$ is the image of $\PsAut(X/S,\De)$ in ${\rm GL}(N^1(X/S)_\Rr/W)$. By Proposition \ref{prop: lift cone}, there is a rational polyhedral cone $\Pi \subset \Mov(X/S)_+$ such that $\Gamma_B \cdot \Pi = \Mov(X/S)_+$, and for each $\gamma \in \Gamma_B$, either $\gamma \Pi\cap \Int(\Pi) = \emptyset$ or $\gamma \Pi = \Pi$. Thus, $\Pi$ is a weak rational polyhedral fundamental domain under the action of $\Gamma_B$.
\end{proof}

\begin{corollary}\label{cor: surface fibration fundamental domain}
Let $f: X \to S$ be a terminal Calabi-Yau fibration. Suppose that $\dim(X/S) \leq 2$, then $\Mov(X/S)_+$ admits a weak rational polyhedral fundamental domain under the action of $\Gamma_B$.
\end{corollary}
\begin{proof}
As mentioned in the proof of Corollary \ref{cor: surface fibration}., when $\dim(X_\eta)\leq 2$, there exists a rational polyhedral cone $P_\eta \subset \Eff(X_\eta)$ such that 
\[
 \PsAut(X_\eta)\cdot P_\eta =  \bMov(X_\eta) \cap \Eff(X_\eta) \supset \Mov(X_\eta).
\] Besides, it is known that good minimal models exist for effective klt pairs in dimension $\leq 3$. Then the claim follows from Theorem \ref{thm: weak fundamental domains}.
\end{proof}

\bibliographystyle{alpha}
\bibliography{bibfile}
\end{document}